\documentclass[a4paper, reqno, 10pt]{amsart}

\usepackage{a4wide}
\usepackage{verbatim}
\usepackage[dvips]{color}
\usepackage[dvipsnames]{xcolor}
\usepackage{amssymb}
\usepackage{amsfonts}
\usepackage[active]{srcltx}
\usepackage{bbm}

\usepackage{kotex}

\usepackage[colorlinks,pdfpagelabels,pdfstartview = FitH,bookmarksopen = true,bookmarksnumbered = true,linkcolor = NavyBlue,plainpages = false,hypertexnames = false,citecolor = OliveGreen,pagebackref=false] {hyperref}
\usepackage[italian, textsize=tiny, color=BlueGreen, backgroundcolor=BlueGreen, linecolor=BlueGreen]{todonotes}
\usepackage{mathtools}
\usepackage{empheq}
\usepackage{cite}
\allowdisplaybreaks
\allowdisplaybreaks

\hypersetup{ colorlinks = true, urlcolor = blue, linkcolor = blue, citecolor = red }

\newtheorem{theorem}{Theorem}[section]

\newtheorem{proposition}[theorem]{Proposition}
\newtheorem{lemma}[theorem]{Lemma}

\theoremstyle{definition}
\newtheorem{definition}[theorem]{Definition}

\newtheorem{remark}[theorem]{Remark}

\numberwithin{equation}{section}

\def\eqn#1$$#2$${\begin{equation}\label#1#2\end{equation}}
\def\charfn_#1{{\raise1.2pt\hbox{$\chi_{\kern-1pt\lower3pt\hbox{{$\scriptstyle#1$}}}$}}}

\makeatletter
\newcommand{\pushright}[1]{\ifmeasuring@#1\else\omit\hfill$\displaystyle#1$\fi\ignorespaces}
\newcommand{\pushleft}[1]{\ifmeasuring@#1\else\omit$\displaystyle#1$\hfill\fi\ignorespaces}
\makeatother
\newcommand{\ov}[1]{\mkern 1.5mu\overline{\mkern-1.5mu#1\mkern-1.5mu}\mkern 1.5mu}

\newcommand{\eps}{\varepsilon}

 \DeclareMathOperator*{\osc}{osc}

\def\dist{\operatorname{dist}}

\newcommand{\dive}{\textnormal{div}}

\newcommand{\data}{\texttt{data}}

\def\<{\langle}
\def\>{\rangle}
\def\Tilde{\widetilde}
\DeclarePairedDelimiter{\norm}{\lVert}{\rVert}
\DeclarePairedDelimiter\abs{\lvert}{\rvert}
\def\({\left(}
\def\){\right)}
\def\[{\left[}
\def\]{\right]}

\newcommand{\cal}[1]{\mathcal{#1}}

\def\er{\mathbb R}
\newcommand{\ern}{\mathbb{R}^n}

\def\loc{{\operatorname{loc}}}

\def\R{\mathbb R}
\def\N{\mathbb N}

\newcommand{\dx}{\,dx}

\newcommand{\supp}{{\rm supp}}

\newcommand{\M}{\mathbb{M}}
\newcommand{\mcA}{\mathcal{A}}
\newcommand{\BMO}{\mathrm{BMO}}
\newcommand{\barM}{\ov{\mathbb{M}}}

\def\mean#1{\mathchoice%
          {\mathop{\kern 0.2em\vrule width 0.6em height 0.69678ex depth -0.58065ex
                  \kern -0.8em \intop}\nolimits_{\kern -0.4em#1}}%
          {\mathop{\kern 0.1em\vrule width 0.5em height 0.69678ex depth -0.60387ex
                  \kern -0.6em \intop}\nolimits_{#1}}%
          {\mathop{\kern 0.1em\vrule width 0.5em height 0.69678ex
              depth -0.60387ex
                  \kern -0.6em \intop}\nolimits_{#1}}%
          {\mathop{\kern 0.1em\vrule width 0.5em height 0.69678ex depth -0.60387ex
                  \kern -0.6em \intop}\nolimits_{#1}}}

\def\vintslides_#1{\mathchoice%
          {\mathop{\kern 0.1em\vrule width 0.5em height 0.697ex depth -0.581ex
                  \kern -0.6em \intop}\nolimits_{\kern -0.4em#1}}%
          {\mathop{\kern 0.1em\vrule width 0.3em height 0.697ex depth -0.604ex
                  \kern -0.4em \intop}\nolimits_{#1}}%
          {\mathop{\kern 0.1em\vrule width 0.3em height 0.697ex depth -0.604ex
                  \kern -0.4em \intop}\nolimits_{#1}}%
          {\mathop{\kern 0.1em\vrule width 0.3em height 0.697ex depth -0.604ex
                  \kern -0.4em \intop}\nolimits_{#1}}}

\newcommand{\aveint}[2]{\mathchoice%
          {\mathop{\kern 0.2em\vrule width 0.6em height 0.69678ex depth -0.58065ex
                  \kern -0.8em \intop}\nolimits_{\kern -0.45em#1}^{#2}}%
          {\mathop{\kern 0.1em\vrule width 0.5em height 0.69678ex depth -0.60387ex
                  \kern -0.6em \intop}\nolimits_{#1}^{#2}}%
          {\mathop{\kern 0.1em\vrule width 0.5em height 0.69678ex depth -0.60387ex
                  \kern -0.6em \intop}\nolimits_{#1}^{#2}}%
          {\mathop{\kern 0.1em\vrule width 0.5em height 0.69678ex depth -0.60387ex
                  \kern -0.6em \intop}\nolimits_{#1}^{#2}}}

\def\Xint#1{\mathchoice
    {\XXint\displaystyle\textstyle{#1}}%
    {\XXint\textstyle\scriptstyle{#1}}%
    {\XXint\scriptstyle\scriptscriptstyle{#1}}%
    {\XXint\scriptscriptstyle\scriptscriptstyle{#1}}%
    \!\int}
\def\XXint#1#2#3{\setbox0=\hbox{$#1{#2#3}{\int}$}
    \vcenter{\hbox{$#2#3$}}\kern-0.5\wd0}
\def\bint{\Xint-}
\def\dashint{\Xint{\raise4pt\hbox to7pt{\hrulefill}}}

\newtoks\by
\newtoks\paper
\newtoks\book
\newtoks\jour
\newtoks\yr
\newtoks\pages
\newtoks\vol
\newtoks\publ

\def\ota{{\hbox{\bf ???}}}
\def\cLear{\by=\ota\paper=\ota\book=\ota\jour=\ota\yr=\ota
\pages=\ota\vol=\ota\publ=\ota}
\def\endpaper{\the\by, \textit{\the\paper},
{\the\jour} \textbf{\the\vol} (\the\yr), \the\pages.\cLear}
\def\endbook{\the\by, \textit{\the\book},
\the\publ, \the\yr.\cLear}
\def\endpap{\the\by, \textit{\the\paper}, \the\jour.\cLear}
\def\endproc{\the\by, \textit{\the\paper}, \the\book, \the\publ,
\the\yr, \the\pages.\cLear}

\author[S.-S. Byun]{Sun-Sig Byun}
\address{Department of Mathematical Sciences and Research Institute of Mathematics,
Seoul National University \\
Seoul 08826, Republic of Korea}
\email{byun@snu.ac.kr}

\author[Y. Cho]{Yumi Cho}
\address{Research Institute of Mathematics, Seoul National University \\
Seoul 08826, Republic of Korea}
\email{imuy31@snu.ac.kr}

\author[S. Ryu]{Seungjin Ryu}
\address{Department of Mathematics, University of Seoul \\
Seoul 02504, Republic of Korea}
\email{seungjinryu@uos.ac.kr}

\thanks{
This work was supported by the National Research Foundation of Korea(NRF) grant funded by the Korea government(MSIT) 
(No. NRF-2022R1A2C1009312 (S.-S. Byun), 
RS-2022-NR074838 (Y. Cho), and
NRF-2020R1C1C1A01014310 (S. Ryu)).
}

\makeatletter
\@namedef{subjclassname@2020}{\textup{2020} Mathematics Subject Classification}
\makeatother


\begin{document}

\title[]{Calder\'{o}n-Zygmund estimates for double phase problems with matrix weights}

\begin{abstract}
We establish an optimal Calder\'{o}n-Zygmund theory for nonuniformly elliptic double phase problems with matrix weights. For $1<p<q<\infty$, $a(\cdot)\in C^{0,\alpha}(\Omega)$ ($0<\alpha\le1$), 
and a symmetric, almost everywhere positive definite matrix weight $\M$ with $|\M(x)|\,|\M(x)^{-1}|\le\Lambda$ for some constant $\Lambda\ge 1$ and small $|\log \M|_{\mathrm{BMO}}$, we prove, for every $\gamma>1$,
$$
(|\M F|^p+a(x)|\M F|^q)\in L^\gamma_{\mathrm{loc}}
\;\Longrightarrow\;
(|\M Du|^p+a(x)|\M Du|^q)\in L^\gamma_{\mathrm{loc}}.
$$
Our argument combines a freezing of the logarithm of the matrix field, $\log \M$, with a fractional maximal-operator method governed by the Muckenhoupt-Wheeden $\mathcal{A}_{p,s}$ classes (where $1/s=1/p-\alpha/(nq)$). This yields scale-invariant comparison and level-set estimates and precludes Lavrentiev gaps at the sharp threshold $q/p\le 1+\alpha/n$. Our result recovers the identity case $\,\M\equiv {\rm I}_n\,$, i.e., the classical (unweighted) Calder\'{o}n-Zygmund theory for double-phase problems.
\end{abstract}

\maketitle

\section{Introduction}


In this work we study the double phase problem involving a degenerate/singular matrix-valued weight whose prototype is 
\begin{multline}
\label{0000}
\dive\(|\M Du|^{p-2}\M^2 Du+a(x)|\M Du|^{q-2}\M^2 Du\)\\
=
\dive\(|\M F|^{p-2}\M^2 F+a(x)|\M F|^{q-2}\M^2 F\)
\quad\mbox{in $\Omega$,}
\end{multline}
where $\Omega$ is a bounded domain in $\R^n$ with $n\ge2$.
The numbers $p$, $q$ and the function $a:\Omega\rightarrow\R$ interplay
with each other under a basic structural assumption
that
\begin{align}\label{holder-a} 1<p<q<\infty, \quad 0\le
a(\cdot)\in C^{0,\alpha}(\Omega) \quad\mbox{with $\alpha\in(0,1]$.}
\end{align}
Here $F:\Omega\rightarrow\R^n$ is a given vector field and 
$\M:\ern\rightarrow\er^{n\times n}$ is a matrix-valued weight that is symmetric, almost everywhere positive definite, and satisfies the quantitative anisotropy bound
\begin{align}\label{cond-M}
|\mathbb{M}(x)||\mathbb{M}^{-1}(x)|\leq\Lambda
\end{align}
for almost all $x\in\R^n$ and some constant $\Lambda\ge1$,
where $\abs{\cdot}$ is the usual spectral norm of matrices. 
This condition implies that for any $\xi \in
\mathbb{R}^n$ there holds
\begin{align*}
\Lambda^{-1}\abs{\M}\abs{\xi} \le \abs{\M\xi} \le \abs{\M}\abs{\xi}.
\end{align*}
With $\mathbb{M}={\rm I}_n$, 
Zhikov introduced in \cite{Z1, Z2, MR1350506} the double-phase problem, motivated by the modeling of strongly anisotropic materials, elasticity, homogenization, and the Lavrentiev phenomenon.
In particular, the sharp
threshold
\begin{equation}
\label{lavrentiev}
\frac{q}{p} \leq 1+\frac{\alpha}{n}
\end{equation} 
is linked to the absence of the Lavrentiev phenomenon 
and serves as a necessary regularity condition for achieving optimal results 
such as gradient H\"{o}lder continuity and Calder\'{o}n-Zygmund estimates.
These aspects were rigorously developed in a series of foundational works \cite{BCM, BOh1, CM1, CM2, CM3, MR3985927, MR2076158}, which have significantly influenced subsequent research; see, for example, \cite{BB1, HO1, MR4897661, MR4627284, MR4926894, MR4693679, BM1, DM1, MR2058167, MR4361836, MR4258791}.

The purpose of this paper is to establish a Calder\'on-Zygmund type theory for the matrix-weighted double phase problems. 
More precisely, we prove that for every $\gamma>1$, the following implication
\begin{align*}
\big(|\M F|^p+a(x)\,|\M F|^q\big)\in L^\gamma_{\mathrm{loc}}
\ \Longrightarrow\
\big(|\M Du|^p+a(x)\,|\M Du|^q\big)\in L^\gamma_{\mathrm{loc}},
\end{align*}
holds with quantitative dependence on $\Lambda$ and $|\log \M|_{\mathrm{BMO}}$. 
When $\M={\rm{I}}_n$, this was proved in \cite{MR3985927,CM3}.
Especially, a notable contribution of the paper
\cite{MR3985927} is a detailed treatment of the borderline case in
particular concerning the absence of the Lavrentiev phenomenon in
the context of the Calder\'{o}n-Zygmund theory for nonuniformly
elliptic problems. Indeed the assertion was proved
under the assumption $\frac{q}{p}<1+\frac{\alpha}{n}$ in earlier
work \cite{CM3}. Based on the approach considered in \cite{CM3}, the
authors of \cite{MR3985927} successfully handled the borderline case
by employing an improved analytic technique of fractional estimates
developed in \cite{CM1}.
We generalize this framework by incorporating nontrivial matrix weights, thereby capturing a broader class of anisotropic behaviors.
The models with variable matrix weights have typically been studied in single phase problems $(a\equiv0)$,
where small $|\log \M|_{\mathrm{BMO}}$ or Muckenhoupt weights are used to obtain quantitative estimates in both linear and nonlinear cases, see, for instance, \cite{MR4410267,MR4629762,BCL,BY1,MR4926899,MR3900368}. 
Building on these developments, the present work bridges these lines of research by treating the full double phase structure with a variable matrix weight.

Two novel structural features arise in this work.
First, the $q$-phase with a H\"older-regular coefficient brings in a fractional maximal operator in weighted Lebesgue spaces, necessitating control via the fractional Muckenhoupt-Wheeden  $\mathcal{A}_{p,s}$ condition (see, e.g., \cite{MR293384,MR340523,MR930072,MR1247202,MR369641}).
Such a phenomenon is absent in single phase or identity-weighted problems. 
The classical Muckenhoupt class characterizes the boundedness of the Hardy–Littlewood maximal operator on weighted Lebesgue spaces, while the fractional Muckenhoupt-Wheeden class governs the boundedness of fractional maximal operators, which is naturally applied to the analysis of the $q$-phase in weighted double phase problems.
The class $\mathcal{A}_{p,s}$ with the parameter relation $1/s=1/p-\alpha/(nq)$ plays a crucial role in preventing the Lavrentiev phenomenon in our setting, as connected to the sharp threshold \eqref{lavrentiev}.
The second is small $\log$-BMO control of $\M$.
The smallness of $|\log \M|_{\mathrm{BMO}}$ enables a perturbative approach, allowing us to locally freeze the matrix weight and control deviations through logarithmic BMO estimates.
The method, previously adopted in single-phase problems, is here extended to the double phase problems.
We briefly outline this mechanism. 
We work locally and freeze $\M$ on each ball $B$ by its logarithmic mean $(\M)^{\log}_B$. 
The solution is then compared with four reference problems in a nested way:
the homogeneous problem with the original nonlinearity, 
one with the frozen matrix $(\M)^{\log}_B$,
one with the $x$-dependence of the nonlinearity removed by averaging over $B$, and 
one with a frozen height $a_0$ for the $q$-phase on a slightly smaller ball.
The transition from the variable matrix weight $\M$ to the frozen weight $(\M)^{\log}_B$ can be controlled by the smallness of $|\log\M|_{\BMO}$. 
Once the weight is frozen, the analysis proceeds along the lines of weight-free double phase problems,
but particular care is taken to ensure that the resulting estimates remain independent of the size of the constant matrix $(\M)^{\log}_B$.
Subsequently, a Vitali covering and exit-time argument on the super level sets of $H(x,\M Du)=|\M Du|^p+a(x)\,|\M Du|^q,$
together with pointwise control from the frozen problem and the smallness of $\M-(\M)^{\log}_B$, produce a good-$\lambda$ inequality. 
Iterating this inequality gives the $L^\gamma_{\mathrm{loc}}$-improvement for all $\gamma>1$.

This work provides the first comprehensive regularity theory for double phase problems with general matrix weights, bridging the gap between classical double phase problems and anisotropic models induced by variable matrices.
The two ingredients - the fractional $\mathcal{A}_{p,s}$ and the $\log$-BMO smallness of $\M$ - prevent the Lavrentiev gap and recover the classical results as the limit $\M\to {\rm I}_n$.

Our paper is organized as follows. Section \ref{sec2} recalls weighted background (including $\mathcal{A}_{p,s}$) and states the main result. Section \ref{sec3} treats the absence of the Lavrentiev phenomenon in the matrix-weighted setting and specifies admissible solutions. Section \ref{sec4} develops higher integrability for the reference problems (with frozen $a_0$ and frozen $(\M)_B^{\log}$). Section \ref{sec5} proves the main theorem by combining the comparison estimates with a covering/level-set iteration.

We close this introduction by noting that the analysis presented here naturally suggests several important directions for future research, including the extension of this matrix-weighted theory to the more general Musielak-Orlicz setting and the investigation of the boundary regularity.

\section{preliminaries and main results}\label{sec2}

In this section we present fundamental analytic tools to treat
matrix-valued weights, and then introduce regularity assumptions
regarding our degenerate/singular double phase problem to state the
main result.
\subsection{Weights and auxiliary lemmas}
We start with the notations to be used. We write
$B_r(x_0)\subset\R^n$ to denote the open ball with center
$x_0\in\R^n$ and radius $r>0$. We often omit specifying the
center point of a ball when it is clear in the context. For any
$p\in(1,\infty)$ we write $p'=\frac{p}{p-1}$.
We write
$(f)_B=\bint_B f \ dx$ for the integral average of a function $f$
over a ball $B$. Moreover $\mathbbm{1}_D$ denotes the indicator
function of a nonempty set $D\subset\Omega$. We write
$$
[a]_{C^{0,\alpha}(D)}=\sup_{\substack{x,y\in D \\ x\neq
y}}\frac{|a(x)-a(y)|}{|x-y|^\alpha}
$$
and $[a]_{0,\alpha} = [a]_{C^{0,\alpha}(\Omega)}$.
We also write $H(x,z)=|z|^p+a(x)|z|^q$.

We now briefly review some properties on weights. We say
that $\mu:\R^n\rightarrow[0,\infty)$ is a (scalar) weight if it is
almost everywhere positive. For any number $p\in(1,\infty)$ the
$\mathcal{A}_p$-Muckenhoupt class is the set of all weights $\mu\in
L^1_{\loc}(\R^n)$ for which
\begin{align}\label{201-a}
[\mu]_{\mathcal{A}_p}=
\sup_{x_0\in\R^n}\sup_{r>0}\bigg(\bint_{B_r(x_0)}\mu\,dx\bigg)
\bigg(\bint_{B_r(x_0)}\mu^{-\frac{1}{p-1}}\,dx\bigg)^{p-1} <\infty.
\end{align}
We denote by $\mathcal{A}_p(D)$ to mean this Muckenhoupt class
if the supremum in \eqref{201-a} is taken over all balls $B$ such
that $B\subset D$, where $D$ is a bounded domain in $\R^n$. A
typical example of a weight in $\cal{A}_p$ is the function
$\abs{x}^{\sigma}$ with $-n<\sigma<n(p-1)$. We recall here basic
properties for $\cal{A}_p$-class from \cite{MR1232192, MR293384,
Gra}.

\begin{lemma}
\label{lem:Ap-prop} Let $1<p<\infty$ be given. Then
\begin{enumerate}
\item If $\mu\in \mcA_p$, then $\mu^{1-p'}\in\mcA_{p'}$
and 
$[\mu^{1-p'}]_{\mcA_{p'}}=[\mu]_{\mcA_p}^{p'-1}$.
\item If $p<q<\infty$, then
$\mcA_p\subset\mcA_q$.
\item If $\mu\in\mcA_p$, then there holds for any ball $B\subset\R^n$
\begin{align*}
\bint_B\mu\,dx
\le
[\mu]_{\mcA_p}
\exp\bigg(\bint_B\log\mu\,dx\bigg).
\end{align*}

\item If $\mu\in \mcA_p$,
then there is a constant $c=c(n,p,[\mu]_{\cal{A}_p})>0$ such that
\begin{align*}
\frac{1}{c}\int_{D}|f|^p\mu\,dx
\le
\int_{D} (\cal{M}[\mathbbm{1}_Df])^p\mu\,dx
\le
c\int_{D} |f|^p\mu\,dx
\end{align*}
for any measurable function $f$ with $|f|^p\mu\in L^1(D)$,
where
\begin{align}\label{hardy-max}
\cal{M}[\mathbbm{1}_D f](x) =
\sup_{r>0}\bint_{B_r(x)}|\mathbbm{1}_Df(y)|\,dy.
\end{align}
\end{enumerate}
\end{lemma}

We introduce the $\mathcal{A}_{p,s}$-class which is a
generalization of the classical Muckenhoupt $\cal{A}_{p}$, and plays
a crucial role to prove the absence of Lavrentiev phenomena in our
problem. We say that weight $\mu$ satisfies the
$\cal{A}_{p,s}$-condition for $p, s\in(1,\infty)$, denoted by
$\mu\in \mathcal{A}_{p,s}$, if
\begin{align}\label{frac-weight}
[\mu]_{\cal{A}_{p,s}} =
\sup_{x_0\in\R^n}\sup_{r>0}\bigg(\bint_{B_r(x_0)}\mu\,dx\bigg)
\bigg(\bint_{B_r(x_0)}\mu^{-p'/s}\,dx\bigg)^{s/p'}<\infty.
\end{align}
It follows that $\mu\in \mathcal{A}_{p,s}$ if and only if $\mu\in
\mathcal{A}_\sigma$ for $\sigma=1+\frac{s}{p'}$.

\begin{lemma}[\cite{MR340523}]\label{lem:Aps}
Let $\theta \in(0, n)$ be given and let $\mu\in L_{\loc}^1(\R^n)$ be a weight. If
$\mu^s \in \mathcal{A}_{p,s}$ for some $1<p<\frac{n}{\theta}$ and
the number $s = \frac{np}{n-\theta p}$, then there holds
\begin{align*}
\bigg(\int_{D}\(\cal{M}_\theta
[\mathbbm{1}_Df]\)^s\mu^s\,dx\bigg)^{\frac{1}{s}} \le
c\bigg(\int_{D}|f|^p\mu^p\,dx\bigg)^{\frac{1}{p}},
\end{align*}
for some constant $c$ independent of $f$ and for any measurable function
$f$ with $f\mu\in L^p(D)$, where
\begin{align}\label{frac-max}
\cal{M}_\theta [\mathbbm{1}_D f](x) =
\sup_{r>0}r^{\theta-n}\int_{B_r(x)}| \mathbbm{1}_D f(y)|\ dy.
\end{align}
\end{lemma}

We now recall the matrix weight $\M:\R^n \to \R^{n\times n}$ and
write $\omega=|\M|$ to mean the corresponding scalar-valued weight.
We introduce the weighted Lebesgue spaces related to the weight
$\omega$. Under the assumption $\omega\in L^p_{\loc}(\R^n)$ the
weighted Lebesgue space $L^p_{\omega^p}(\Omega)$ consists of all
measurable functions $f:\Omega\rightarrow\R$ such that $\omega f\in
L^p(\Omega)$ with its norm
$\norm{f}_{L^p_{\omega^p}(\Omega)}=\norm{f\omega}_{L^p(\Omega)}$.
The weighted Sobolev space $W^{1,p}_{\omega^p}(\Omega)$ is the one
consisting of all functions $f\in W^{1,1}(\Omega)$ such that
$f,\abs{Df}\in L^p_{\omega^p}(\Omega)$ with the norm
$\norm{f}_{W^p_{\omega^p}(\Omega)}=\norm{f}_{L^p_{\omega^p}(\Omega)}+\norm{\abs{Df}}_{L^p_{\omega^p}(\Omega)}$.
For the sake of simplicity and clarity we reformulate the
$\mathcal{A}_p$-Muckenhoupt class associated to the weight
$\omega^p$ as follow: $\omega^p\in\cal{A}_p$ if and only if
\begin{align*}
[\omega^p]_{\mathcal{A}_p}^{\frac{1}{p}}=
\sup_{x_0\in\R^n}\sup_{r>0}\bigg(\bint_{B_r(x_0)}\omega^p\,dx\bigg)^{\frac{1}{p}}
\bigg(\bint_{B_r(x_0)}\omega^{-p'}\,dx\bigg)^{\frac{1}{p'}}
<\infty.
\end{align*}

We next present a certain Sobolev-Poincar\'{e}'s inequality with
multiplicative weights.

\begin{lemma}[\cite{MR4410267,MR4629762}]
\label{lem:weight-poin} Let $1<p<\infty$ and $\theta \in(0,1]$ with
$\theta p\ge\max\left \{1,\frac{np}{n+p} \right\}$.
Setting $\omega=|\M|$, suppose \eqref{cond-M} and
\begin{align*}
\sup_{B'\subset B_{2r}}\bigg(\bint_{B'}\omega^p\,dx\bigg)^{\frac{1}{p}}
\bigg(\bint_{B'}\omega^{-(\theta p)'}\,dx\bigg)^{\frac{1}{(\theta p)'}}
\le
c_0
\end{align*}
for some $c_0>0$, where the supremum is taken over all balls $B'$ contained in $B_{2r}$.
Then we have
\begin{align*}
\bigg(\bint_{B_{r}}\bigg|\frac{\M (v-(v)_{B_{r}})}{r}\bigg|^p\,dx\bigg)^{\frac{1}{p}}
\le
c
\bigg(\bint_{B_{r}}|\M Dv|^{\theta p}\,dx\bigg)^{\frac{1}{\theta p}},
\end{align*}
where $c=c(c_0,n,p,\Lambda)>0$. 
Furthermore, 
if $v$ vanishes on
$\partial B_r\cap B_\rho(y)$,
where $B_\rho(y)$ is a ball such that $y\in B_r$, $\partial B_r\cap B_\rho(y)\neq\emptyset$ and
$|B_\rho(y) \setminus B_r|\ge \nu_0|B_\rho|$ for some constant $\nu_0>0$,
then we have
\begin{align*}
\bigg(\frac{1}{|B_\rho|}\int_{ B_\rho(y)\cap B_{r}}\bigg|\frac{\M
v}{\rho}\bigg|^p\,dx\bigg)^{\frac{1}{p}} \le c
\bigg(\frac{1}{|B_\rho|}\int_{B_\rho(y)\cap B_{r}}|\M Dv|^{\theta
p}\,dx\bigg)^{\frac{1}{\theta p}}
\end{align*}
for some constant $c=c(c_0,n,p,\Lambda,\nu_0)>0$.
\end{lemma}

Using Lemma~\ref{lem:weight-poin} we can obtain the following weighted norm estimate.

\begin{lemma}
\label{lem:Hsobolev} 
Let $0<r\le1$, $1<p<q<\infty$ and
$0<a(\cdot)\in C^{0,\alpha}(B_r)$ with \eqref{lavrentiev}.
Assume \eqref{cond-M} and set $\omega=|\M|$.
Then there are constants $s_o\in(p',\infty)$ and $\theta_o\in(0,1)$, depending only on $n,p$ and $q$,
such that if
\begin{align}\label{cond-M-xxx}
\sup_{B'\subset B_{2r}}\bigg(\bint_{B'}\omega^q\,dx\bigg)^{\frac{1}{q}}
\bigg(\bint_{B'}\omega^{-s_o}\,dx\bigg)^{\frac{1}{s_o}}
\le
c_1
\end{align}
for some $c_1>0$, where the supremum is taken over all balls $B'$ contained in $B_{2r}$,
then we have
\begin{align}\label{Hsobolev-1}
\bint_{B_r}H\bigg(x,\frac{\M( v- (v)_{B_r})}{r}\bigg)\,dx
\le 
c\(1+\|\M Dv\|_{L^{p}(B_r)}^{q-p}\)\bigg(\bint_{B_r} [H(x,\M Dv)]^{\theta_o}\,dx\bigg)^{\frac{1}{\theta_o}}
\end{align}
for some $c=c(c_1, n,p,q,\Lambda,[a]_{C^{0,\alpha}(B_r)})>0$.
Furthermore, if $v$ vanishes on $\partial B_r\cap B_\rho(y)$,
where $B_\rho(y)$ is a ball such that $y\in B_r$, $\partial B_r\cap B_\rho(y)\neq\emptyset$ and
$|B_\rho(y) \setminus B_r|\ge \nu_0|B_\rho|$ for some constant $\nu_0>0$,
then we have
\begin{align}\label{Hsobolev-2}
\frac{1}{|B_\rho|}\int_{B_\rho(y)\cap B_r}\hspace{-0.2cm}
H\(x,\frac{\M v}{\rho}\)\,dx
\le 
c  \(\hspace{-0.05cm}1+\|\M Dv\|_{L^{p}(B_r)}^{q-p}\hspace{-0.05cm}\)\hspace{-0.05cm}
\bigg(\frac{1}{|B_\rho|}\int_{B_\rho(y)\cap B_r}\hspace{-0.2cm}
 [H(x,\M Dv)]^{\theta_o}\,dx\bigg)^{\frac{1}{\theta_o}}
\end{align}
for some constant $c=c(c_1, n,p,q,\Lambda,[a]_{C^{0,\alpha}(B_r)},\nu_0)>0$.
\end{lemma}
\begin{proof}
As the proof of \eqref{Hsobolev-2} proceeds in the same manner to that of \eqref{Hsobolev-1}, 
we shall only prove \eqref{Hsobolev-1}.
To estimate the $q$-phase term, we divide the analysis into two cases based on the behavior of the modulating coefficient $a$, 
adopting the method used in the proof of \cite[Lemma~3.1]{MR4361836}.

We first assume that $\sup_{B_r}a(\cdot)> 4[a]_{0,\alpha}r^\alpha$.
Then there is a point $y\in B_r$ such that
\begin{align*}
a(y)\ge4[a]_{0,\alpha} r^{\alpha},
\end{align*} 
which follows that
\begin{align}\label{xx-209}
|a(x)-a(y)|\le 2 [a]_{0,\alpha}r^\alpha
\le\tfrac{1}{2}a(y)
\quad\mbox{for all $x\in B_r$}.
\end{align}
This implies
\begin{align*}
a(y)\le |a(x)-a(y)|+a(x)
\le \tfrac{1}{2}a(y)+a(x).
\end{align*}
Absorbing $\frac{1}{2}a(y)$ from the right-hand side into the left-hand side, we obtain
\begin{align}\label{xx-210}
a(x)\ge\tfrac{1}{2}a(y)
\quad\mbox{for all $x\in B_r$}.
\end{align}
It follows from \eqref{xx-209} that 
\begin{align}\label{xx-211}
a(x)\le |a(x)-a(y)|+a(y)\le 2a(y)
\quad\mbox{for all $x\in B_r$}.
\end{align}
Let us choose $\theta\in(0,1)$ such that
\begin{align*}
\max\left\{1,\frac{np}{n+p}\right\}<\theta p.
\end{align*}
We have
$\max\left\{1,\frac{nq}{n+q}\right\}<\theta q$,
and hence we apply \eqref{cond-M-xxx} and Lemma~\ref{lem:weight-poin}, together with \eqref{xx-210} and \eqref{xx-211}, to deduce
\begin{align}\label{xxx-213-a}
\bint_{B_r}a(x)\frac{|\M(v-(v)_{B_r})|^q}{r^q}\,dx
&\le
ca(y)
\bigg(\bint_{B_r}|\M Dv|^{\theta q }\,dx\bigg)^{\frac{1}{\theta}}\nonumber\\
&\le
c
\bigg(\bint_{B_r}\[(a(x)|\M Dv|^{q }\]^{\theta}\,dx\bigg)^{\frac{1}{\theta}}
\end{align}
for some $c=c(c_1, n,q,\Lambda)>0$, provided $s_o\ge(\theta q)'$.
Similarly, for $p$-phase term when $s_o\ge(\theta p)'$, we have
\begin{align}\label{xxx-213}
\bint_{B_r}\frac{|\M(v-(v)_{B_r})|^p}{r^p}\,dx
&\le
c\bigg(\bint_{B_r}|\M Dv|^{\theta p }\,dx\bigg)^{\frac{1}{\theta}}
\end{align}
for some $c=c(c_1, n,p,\Lambda)>0$.
Combining \eqref{xxx-213-a} and \eqref{xxx-213} yields  
\begin{align}\label{xxx-217-b}
\bint_{B_r}H\(x,\frac{|\M(v-(v)_{B_r})|}{r}\)\,dx
&\le
c\bigg(\bint_{B_r}[H(x,\M Dv)]^{\theta}\,dx\bigg)^{\frac{1}{\theta}},
\end{align}
whenever $s_o\ge(\theta p)'>(\theta q)'$.

We next assume that $\sup_{B_r}a(\cdot)\le 4[a]_{0,\alpha}r^\alpha$.
Observing  that $\max\{1,\frac{nq}{n+q}\}< p$ from \eqref{lavrentiev},
we choose $\tilde{\theta}\in(0,1)$ such that
\begin{align*}
\max\left\{1,\frac{nq}{n+q}\right\}<\tilde{\theta} q<p.
\end{align*}
We employ \eqref{cond-M-xxx}, Lemma~\ref{lem:weight-poin} and H\"older's inequality,
to have
\begin{align}\label{xxx-207}
\bint_{B_r}a(x)\frac{|\M(v-(v)_{B_r})|^q}{r^q}\,dx
&\le
4[a]_{0,\alpha}r^\alpha
\bigg(\bint_{B_r}|\M Dv|^{\tilde{\theta}q }\,dx\bigg)^{\frac{1}{\tilde{\theta}}}\nonumber\\
&\le
cr^\alpha
\bigg(\bint_{B_r}|\M Dv|^{\tilde{\theta} q}\,dx\bigg)^{\frac{q-p}{\tilde{\theta} q}}
\bigg(\bint_{B_r}|\M Dv|^{\tilde{\theta} q}\,dx\bigg)^{\frac{p}{\tilde{\theta} q}}\nonumber\\
&\le
cr^\alpha
\bigg(\bint_{B_r}|\M Dv|^{p}\,dx\bigg)^{\frac{q-p}{p}}
\bigg(\bint_{B_r}|\M Dv|^{\tilde{\theta} q}\,dx\bigg)^{\frac{p}{\tilde{\theta} q}}\nonumber\\
&\le
cr^{\alpha-\frac{(q-p)n}{p}}
\|\M Dv\|_{L^p(B_r)}^{q-p}
\bigg(\bint_{B_r}|\M Dv|^{\tilde{\theta} q}\,dx\bigg)^{\frac{p}{\tilde{\theta} q}}\nonumber\\
&\le
c
\|\M Dv\|_{L^p(B_r)}^{p-q}
\bigg(\bint_{B_r}|\M Dv|^{\tilde{\theta} q}\,dx\bigg)^{\frac{p}{\tilde{\theta} q}}
\end{align}
for some $c=c(c_1,n,p,q,\Lambda,[a]_{C^{0,\alpha}(B_r)})>0$,
provided $s_o\ge(\tilde{\theta} q)'$.
In the last inequality, we used $r\le1$ and \eqref{lavrentiev}.
It follows from H\"older's inequality and Lemma~\ref{lem:weight-poin} that
\begin{align*}
\bint_{B_r}\frac{|\M(v-(v)_{B_r})|^p}{r^p}\,dx
&\le
\(\bint_{B_r}\frac{|\M(v-(v)_{B_r})|^q}{r^q}\,dx\)^{\frac{p}{q}}\nonumber\\
&\le
c\bigg(\bint_{B_r}|\M Dv|^{\tilde{\theta} q}\,dx\bigg)^{\frac{p}{\tilde{\theta} q}}.
\end{align*}
for some $c=c(c_1,n,p,q,\Lambda)>0$.
Combining this and \eqref{xxx-207} we obtain
\begin{align}\label{220-xxx}
\bint_{B_r}H\(x,\frac{|\M(v-(v)_{B_r})|}{r}\)\,dx
&\le
c\(1+\|\M Dv\|_{L^p(B_r)}^{p-q}\)
\bigg(\bint_{B_r}|\M Dv|^{\tilde{\theta} q}\,dx\bigg)^{\frac{p}{\tilde{\theta} q}}\nonumber\\
&\le
c\(1+\|\M Dv\|_{L^p(B_r)}^{p-q}\)
\bigg(\bint_{B_r}[H(x,\M Dv)]^{\frac{\tilde{\theta}{q}}{p}}\,dx\bigg)^{\frac{p}{\tilde{\theta} q}},
\end{align}
for some $c=c(c_1,n,p,q,\Lambda,[a]_{C^{0,\alpha}(B_r)})>0$, whenever $s_o\ge(\tilde{\theta} q)'$.

In both cases, we select
$$
s_o=\max\{(\theta p)',(\tilde{\theta} q)'\}\in(p',\infty) \quad\mbox{and}\quad
\theta_o=\max\left\{\theta,\tfrac{\tilde{\theta}q}{p}\right\}\in(0,1).
$$
Here, we remark that $s_o<\infty$, since $(\theta p)', (\tilde{\theta} q)'<\infty$ from $\theta p,\tilde{\theta}q>1$.
Therefore we complete the proof by using H\"older's inequality in \eqref{xxx-217-b} and \eqref{220-xxx}.
\end{proof}

We now present a main regularity assumption on the matrix weight
$\mathbb{M}$. We start with $\BMO$ space which is the set of all
functions $f\in L^1_{\loc}(\R^n)$ for which
\begin{align}\label{215-c}
|f|_{\BMO}= \sup_{x_0\in\R^n}\sup_{r>0}
\bigg(\bint_{B_r(x_0)}|f(x)-(f)_{B}|\,dx\bigg) <\infty.
\end{align}
We write $|f|_{\BMO(D)}$ when the supremum in \eqref{215-c} is taken
over all balls $B$ such that $B\subset D$, where $D$ is a bounded
domain in $\R^n$. We next consider the matrix exponential
$\exp:\R^{n\times n}_{\rm{sym}}\rightarrow\R^{n\times n}_{>0}$ and
its inverse $\log:\R^{n\times n}_{>0}\rightarrow\R^{n\times
n}_{\rm{sym}}$, where $\R^{n\times n}_{\rm{sym}}$ and $\R^{n\times
n}_{>0}$ mean the sets of symmetric real-valued matrices and
symmetric positive definite real-valued matrices, respectively. 
Then we can define the logarithmic means of 
the matrix-valued weight $\M$ and its scalar weight
$\omega=|\mathbb{M}|$, denoted by
\begin{align*}
\quad \(\M\)_B^{\log}=\exp\(\bint_B\log\M(x)\,dx\)
\quad\mbox{and}\quad\(\omega\)_B^{\log}=\exp\(\bint_B\log\omega(x)\,dx\),
\end{align*}
since $\M$ is almost everywhere positive definite.
Also, observe the properties
\begin{align}\label{ovM}
\begin{cases}
\Lambda^{-1}(\omega)_B^{\log}|\xi|
\le
|(\M)_B^{\log}\xi|
\le
(\omega)_B^{\log}|\xi|
\quad \mbox{for all $\xi\in\R^n$,} \\
\Lambda^{-1}(\omega)_B^{\log}
\le
|(\M)_B^{\log}|
\le
(\omega)_B^{\log}
\end{cases}
\end{align}
under the assumption \eqref{cond-M}, see \cite{MR4410267}.

In what follows we give useful estimates concerning weights $\M$ and 
$\omega$.

\begin{lemma}
[\cite{MR4410267}] \label{lem:matrix-scalar}
We have
\begin{align*}
\bint_{B}|\log\omega(x)-(\log\omega)_B|\,dx
\le
2\bint_{B}|\log\M(x)-(\log\M)_{B}|\,dx
\end{align*}
and hence,
\begin{align*}
|\log\omega|_{\BMO(B)}\le 2|\log\M|_{\BMO(B)}.
\end{align*}
\end{lemma}

\begin{lemma}
[\cite{MR4410267}] \label{lem:m-m} 
There are positive constants $\kappa$ and $c_b$, depending on $n$, such that if
we have
$|\log\M|_{\BMO(B)}\le\frac{\kappa}{s}$ for each $s\ge1$, then
\begin{align*}
\left(\bint_{B}\bigg(\frac{|\M(x)-\(\M\)_B^{\log}|}{|\(\M\)_B^{\log}|}\bigg)^s\,dx\right)^{\frac{1}{s}}
\le
c_b s|\log\M|_{\BMO(B)}.
\end{align*}
This holds with $\omega$ instead of $\M$.
\end{lemma}

\begin{lemma}
[\cite{MR4410267}] \label{lem:w-prop} Let $\kappa$ and $c_b$ be the
numbers given in Lemma \ref{lem:m-m}. Then one can find a constant
$\beta=\beta(n)=\min\left\{\kappa, \frac{1}{c_b} \right\}$ such that
the followings hold.
\begin{enumerate}
\item If $|\log\omega|_{\BMO(B)}\le\frac{\beta}{s}$ for each $s\ge1$,
then
\begin{align*}
\bigg(\bint_B\omega^s\,dx\bigg)^{\frac{1}{s}}
\le
2\(\omega\)_B^{\log}
\quad\mbox{and}\quad
\bigg(\bint_B\omega^{-s}\,dx\bigg)^{\frac{1}{s}}
\le
\frac{2}{\(\omega\)_B^{\log}}.
\end{align*}

\item 
If $|\log\omega|_{\BMO(B)}\le\beta\min \left\{\frac{1}{p},
\frac{1}{(\theta p)'} \right\}$ for each $p \in (1, \infty)$ and
$\theta\in(0,1]$ with $\theta p>1$, then
$\omega^p\in\cal{A}_{\theta p, p}(B)$ and
\begin{align*}
\sup_{B'\subset B}
\bigg(\bint_{B'}\omega^p\,dx\bigg)^{\frac{1}{p}}
\bigg(\bint_{B'}\omega^{-(\theta p)'}\,dx\bigg)^{\frac{1}{(\theta
p)'}}  \leq 4,
\end{align*}
where the supremum is taken over all balls $B'$ contained in $B$.
\end{enumerate}
\end{lemma}

We end this subsection with the following technical lemma.
\begin{lemma}
[\cite{Giu1}]
\label{lem:st}
Let $r>0$.
Assume that $f:[R/2,R]\rightarrow[0,\infty)$ is a bounded function satisfying
\begin{align*}
f(\rho_1)\le\tau f(\rho_2)+\frac{X}{(\rho_2-\rho_1)^s}
+\frac{Y}{(\rho_2-\rho_1)^t}
\end{align*}
for every $R/2\le\rho_1<\rho_2\le R$, where
$X,Y,s,t\ge0$ and $\tau\in(0,1)$ are fixed parameters.
Then there is a constant $c=c(s,t,\tau)>0$ such that
\begin{align*}
f(R/2)\le c\(\frac{X}{R^s}+\frac{Y}{R^t}\).
\end{align*}
\end{lemma}

\subsection{Main results}
We are now in the position to present our problem and the main
theorem of the present paper. We recall that $\Omega$ is a bounded
open domain in $\R^n$ with $n\ge2$ and that
$\M:\ern\rightarrow\er^{n\times n}$ is a given symmetric and a.e.
positive definite matrix-valued function with \eqref{cond-M} and
$\omega=|\M|$. Then the problem under study is
\begin{align}
\label{eq-u}
\dive\(\M A_p(x,\M Du)+a(x) \M A_q(x,\M Du)\)
=
\dive\, \M G(x,\M F)
\quad\mbox{in $\Omega$,}
\end{align}
where the vector fields $A_p, \ A_q:\Omega\times\R^n\rightarrow\R^n$
are Carath\'{e}odory maps whose derivatives $D_zA_p$ and $D_z A_q$
with respect to second variable again become Carath\'eodory maps
from $\Omega\times\R^n\setminus\{0\}$ to $\R^n$. They are assumed to
satisfy
\begin{align}
\label{ApAq}
\left\{
\begin{array}{l}
|A_p(x,z)|\le L|z|^{p-1}\\
|D_zA_p(x,z)|\le L|z|^{p-2}\\
\nu|z|^{p-2}|\xi|^2\le \<D_zA_p(x,z)\xi,\xi\>
\end{array}\right.
\quad\mbox{and}\quad
\left\{
\begin{array}{l}
|A_q(x,z)|\le L|z|^{q-1}\\
|D_zA_q(x,z)|\le L|z|^{q-2}\\
\nu|z|^{q-2}|\xi|^2\le \<D_zA_q(x,z)\xi,\xi\>
\end{array}\right.
\end{align}
for a.e. $x\in\Omega$ and for all $z\in\R^n\setminus\{0\}$,
$\xi\in\R^n$ and for some $0<\nu\le L<\infty$, while the constants
$p,q$ and the function $a(\cdot)$ are initially given as in
\eqref{holder-a}. Furthermore, $F:\Omega\rightarrow\R^n$ is a given
vector field with
\begin{align*}
H(x,\M F)\in L^1(\Omega)
\end{align*}
and $G:\Omega\times\R^n\rightarrow\R^n$ is a Carath\'eodory map with
the growth condition
\begin{align}\label{G}
|G(x,z)|\le L(|z|^{p-1}+a(x)|z|^{q-1}).
\end{align}
We also recall the notation
\begin{align*}
H(x,z)=|z|^p + a(x)|z|^q \ \ \ \ \  x \in\Omega, \ z \in
\R^n
\end{align*} to observe from \eqref{ApAq} that
\begin{align*}
\<A_p(x,z)+a(x)A_q(x,z),z\>\ge cH(x,z)
\end{align*}
for some constant $c=c(p,q,\nu)$. Here we are interested in a
distributional solution $u$ to \eqref{eq-u} which means that $H(x,\M
Du)\in L^1(\Omega)$ and that for any $\varphi\in C_0^\infty(\Omega)$
we have
\begin{align*}
\int_\Omega \<A_p(x,\M Du),\M D\varphi\>\,dx + \int_\Omega a(x)\<A_q(x,\M
Du),\M D\varphi\>\,dx = \int_\Omega \<G(x,\M F),\M D\varphi\>\,dx.
\end{align*}
Then our goal is to prove the following implication
\begin{align*}
\mbox{$\forall \gamma>1$,}\quad H(x,\M F)\in L^\gamma_\loc(\Omega)
\implies H(x,\M Du)\in L^\gamma_\loc(\Omega)
\end{align*}
with the corresponding Calder\'{o}n-Zygmund estimate under minimal
regularity requirements on the matrix weight $\mathbb{M}$ and the
vector fields $A_p, \ A_q$ with respect to $x$. Note that the case
$\gamma=1$ is the classical as there does not need any extra
regularity assumptions there. On the other hand for our case
$\gamma>1$ we need to impose extra proper assumptions on the
nonlinearities $A_p$, $A_q$ and the matrix weight $\M$, to the basic
structural assumptions mentioned earlier.

Now we introduce our main assumptions.
\begin{definition}
Let $0<\delta<1$. Under the assumptions \eqref{cond-M} and
\eqref{ApAq} we say that $(\log\M, A_p,A_q)$ is $\delta$-vanishing
if there hold
\begin{align*}
|\log\M|_{\BMO(\Omega)}\le\delta
\end{align*}
and
\begin{align*}
\sup_{B\subset\Omega}
\bint_{B}\(\Theta_p+\Theta_q\)[B](x)\,dx\le\delta,
\end{align*}
where $B$ is any ball with $B \subset \Omega$ and
\begin{align}
\label{Theta} \Theta_t[B](x) = \sup_{z\in\R^n\setminus\{0\}}
\frac{|A_t(x,z)-(A_t)_{B}(z)|}{|z|^{t-1}}\le 2L
\end{align}
with $$(A_t)_{B}(z)=\bint_B A_t(x,z)\ dx$$ for $t= p, q$.
\end{definition}

In what follows we shall use the symbol
\begin{align*}
\data\equiv\data(n,p,q,\alpha,\nu,L,\Lambda,\norm{a}_{L^\infty(\Omega)}, [a]_{0,\alpha},\norm{H(x,\M Du)}_{L^1(\Omega)})
\end{align*}
to represent universal constants depending on the known data.
We further define
\begin{align*}
\Omega_{\tau}:=\{x\in\Omega:\dist(x,\partial\Omega)>\tau\}
\end{align*}
for any $\tau>0$.

We are finally ready to state the main result.

\begin{theorem}\label{thm} 
Let $u\in W^{1,1}(\Omega)$ be a
distributional solution to \eqref{eq-u} with $H(x,\M Du)\in
L^1(\Omega)$ and under the assumptions \eqref{holder-a}-\eqref{lavrentiev}, \eqref{ApAq} and \eqref{G}. 
Let $\hat{\tau}>0$ be given. 
Assume that $H(x,\M F)\in L^\gamma(\Omega_{\hat{\tau}})$ for some $\gamma\in(1,\infty)$. 
Then there exist positive small numbers $\delta$ and $r_o$, depending
only on $\data$, $\gamma$, $\hat{\tau}$ and $\|H(\cdot,\M F)\|_{L^{\gamma}(\Omega_{\hat{\tau}})}$, 
such that if $(\log\M,A_p,A_q)$ is $\delta$-vanishing, 
then 
\begin{align*}
\bint_{B_{R/2}}[H(x,\M Du)]^\gamma\,dx
\le
c\(\bint_{B_{R}} H(x,\M Du) \,dx\)^\gamma
+
c\bint_{B_{R}}[H(x,\M F)]^\gamma\,dx
\end{align*}
for all balls $B_R\subset\Omega_{2\hat{\tau}}$ such that  $0<R\le r_o\le1$, 
where $c=c(\data,\gamma,\hat{\tau},\|H(\cdot,\M F)\|_{L^{\gamma}(\Omega_{\hat{\tau}})})>0$.
\end{theorem}

\section{Problem setup}\label{sec3}
\subsection{Absence of Lavrentiev phenomenon}

In this subsection we verify the absence of the Lavrentiev phenomenon
regarding the degenerate/singular problem \eqref{0000} to setup our
problem
\begin{align}\label{func-H}
\cal{H}(v,\Omega) = \int_\Omega H(x,\M Dv) \ dx.
\end{align}
for any $v\in W^{1,1}_{\loc}(\R^n)$ with $H(x, \M Dv)\in
L^1(\Omega)$ under the same structural and regularity requirement as
was identified in \cite{CM1, BCM} to be \eqref{lavrentiev} for the
constant matrix, provided that $\log \mathbb{M}$ is sufficiently
small in the sense of $BMO$.

A natural obstruction in obtaining the interior regularity for the gradient
of a minimizer of the functional  $\cal{H}$ arises when the
associated Lavrentiev phenomenon like
\begin{align}\label{302}
\inf_{v\in X}\cal{H}(v,B)
<
\inf_{v\in Y}\cal{H}(v,B)
\end{align}
occurs for $B\Subset\Omega$, where
$X=W^{1,p}_{\omega^p}(\Omega)$ and $Y=W^{1,p}_{\omega^p}(\Omega)\cap \{v\in W^{1,1}(\Omega):|\M Dv|\in L^q(B)\}$.
As mentioned above, the assumptions \eqref{holder-a} and \eqref{lavrentiev} on $p,q,a$
are an unavoidable condition for the absence of the Lavrentiev phenomenon, 
even when $\M$ is the $n \times n$ identity matrix. 
In general we require an additional condition on $\M$
to avoid the gap in \eqref{302}. The following lemma provides a
sufficient condition on $\M$ to do this.

\begin{lemma}
\label{lem:approx2}
Let $B$ and $\tilde{B}$ be balls such that $B\Subset\Tilde{B}\Subset \Omega$.
Setting $\omega=|\M|$, suppose that \eqref{holder-a}-\eqref{lavrentiev}, and
\begin{align}\label{303}
|\log\M|_{BMO(\Tilde{B})}
\le
\beta\min\left\{\frac{1}{p'},\frac{1}{s}\right\},
\end{align}
where $\frac{1}{s}=\frac{1}{p}-\frac{\alpha}{nq}$ and $\beta$ is
given as in Lemma~\ref{lem:w-prop}. Then for every function $f\in
W^{1,p}_{\omega^p}(\Omega)$ with $H(x,\M Df)\in L^1(\Tilde{B})$, 
there exists a
sequence $\{f_k\}\subset W^{1,\infty}(B)$ such that
\begin{align*}
\mbox{$ f_k\rightarrow f$ in $W^{1,p}_{\omega^p}(B)$ and 
$\displaystyle\int_{B} H(x,\M(x) Df_k(x))\,dx\rightarrow
\int_{B} H(x,\M(x) Df(x))\,dx$}
\end{align*}
as $k\rightarrow\infty$.\end{lemma}
\begin{proof}
We observe from \eqref{lavrentiev}, \eqref{303} and Lemma
\ref{lem:w-prop} that
\begin{align}\label{305}
p<q\le s, \quad
\omega^p\in\cal{A}_p(\Tilde{B}), \quad
\omega^q\in\cal{A}_q(\Tilde{B}),
\quad\mbox{and}\quad
\omega^s\in\cal{A}_{p,s}(\Tilde{B}).
\end{align}

We consider a non-negative standard mollifier $\phi\in C^\infty_0(B_1(0))$ 
with $\int_{\R^n}\phi\,dx=1$. Choosing
$\eps_0\in(0,1)$ so that $B=B_r\Subset
B_{r+\eps_0}\Subset\Tilde{B}\Subset\Omega$, we set
$\phi_\eps(x)=\frac1{\eps^n}\phi(\frac{x}{\eps})$ for $x\in B_\eps(0)$
with $0<\eps<\eps_0$. Then we clearly have $\phi_\eps\in C^\infty_0(B_\eps(0))$, 
$\int_{B_\eps(0)}\phi_\eps\,dx=1$, $0\le\phi_\eps\le
c(n)\eps^{-n}$ and $|D\phi_\eps|\le c(n)\eps^{-(n+1)}$. 
Now we
consider the mollified function $f_\eps(x)=(f\ast\phi_\eps)(x)$ and
introduce the following auxiliary functions
\begin{align*}
a_\eps(x)=\inf_{y\in B_\eps(x)} a(y) \quad\mbox{and}\quad \quad
H_\eps(x,z)=|z|^p+a_\eps(x)|z|^q,
\end{align*}
where $x\in B$ and $z\in\R^n$.
We then have that
\begin{align*}
H(x,\M(x) Df_\eps(x))\rightarrow H(x,\M(x) Df(x)) \quad\mbox{a.e. in $B$.}
\end{align*}
In addition, from the H\"older continuity of $a$ in \eqref{holder-a}, we deduce
that
\begin{align}\label{307}
H(x,\M(x) Df_\eps(x))
&\le
H_\eps(x,\M(x) Df_\eps(x))+
|a(x)-a_\eps(x)||\M(x) Df_\eps(x)|^q\nonumber\\
&\le
H_\eps(x,\M(x) Df_\eps(x))+c[a]_{0,\alpha}\eps^\alpha|\M(x)Df_\eps(x)|^q.
\end{align}
Meanwhile, in light of maximal operators, together with the choice of $\phi_\eps$ and $a_\eps(x)$, we have that
\begin{align}\label{308}
&H_\eps(x,\M(x) Df_\eps(x))\nonumber\\
&\le
|\M(x)|^p\(\int_{B_\eps(x)}|Df(y)|\phi_\eps(x-y)\,dy\)^p
+
a_\eps(x)|\M(x)|^q
\(\int_{B_\eps(x)}|Df(y)|\phi_\eps(x-y)\,dy\)^q\nonumber\\
&\le
c\omega^p(x)\(\bint_{B_\eps(x)}|Df(y)|\,dy\)^p
+
c\omega^q(x)\(\bint_{B_\eps(x)}a^{\frac{1}{q}}(y)|Df(y)|\,dy\)^q\nonumber\\
&\le
c\cal{M}(\mathbbm{1}_{\Tilde{B}}|Df|)^p(x)\omega^p(x)
+
c\cal{M}(\mathbbm{1}_{\Tilde{B}}a^{\frac{1}{q}}|Df|)^q(x)\omega^q(x)
\end{align}
and
\begin{align}\label{309}
\eps^\alpha|\M(x)Df_\eps(x)|^q
\le
c\bigg(\eps^{\frac{\alpha}{q}}\bint_{B_\eps(x)}|Df(y)|\,dy\bigg)^q\omega(x)^q
\le
c[\cal{M}_{\alpha/q}(\mathbbm{1}_{\Tilde{B}}Df)(x)\omega(x)]^q
\end{align}
for a.e. $x\in B$, where $c=c(n,p,q)>0$. Combining
\eqref{307}-\eqref{309}, we discover that
\begin{align*}
&H(x,\M(x) Df_\eps(x)) \\
&\le c
\cal{M}(\mathbbm{1}_{\Tilde{B}}|Df|)^p(x)\omega^p(x)
+ c
\cal{M}(\mathbbm{1}_{\Tilde{B}}a^{\frac{1}{q}}|Df|)^q(x)\omega^q(x)
+
c[\cal{M}_{\alpha/q}(\mathbbm{1}_{\Tilde{B}}Df)(x)\omega(x)]^q
\end{align*}
for a.e. $x\in B$, where $c=c(n, p,q,[a]_{0,\alpha})>0$.

We recall $H(x,\M Df)\in L^1(\Tilde{B})$, \eqref{cond-M},
\eqref{305} and (4) in Lemma~\ref{lem:Ap-prop}, to find that
\begin{align*}
\cal{M}(\mathbbm{1}_{\Tilde{B}}|Df|)^p\omega^p
+
\cal{M}(\mathbbm{1}_{\Tilde{B}}a^{\frac{1}{q}}|Df|)^q\omega^q
\in L^1(\Tilde{B}).
\end{align*}
Moreover, \eqref{305} and Lemma~\ref{lem:Aps} yield that
\begin{align*}
\int_{\Tilde{B}}[\cal{M}_{\alpha/q}(\mathbbm{1}_{\Tilde{B}}Df)(x)\omega(x)]^q\,dx
&\le
c\bigg(\int_{\Tilde{B}}[\cal{M}_{\alpha/q}(\mathbbm{1}_{\Tilde{B}}Df)(x)\omega(x)]^s\,dx\bigg)^{\frac{q}{s}}\nonumber\\
&\le
c\bigg(\int_{\Tilde{B}}|Df|^p\omega(x)^p\,dx\bigg)^{\frac{q}{p}}\nonumber\\
&\le
c\bigg(\int_{\Tilde{B}}|\M Df|^p\,dx\bigg)^{\frac{q}{p}},
\end{align*}
where $c=c(n,p,q,\alpha,\Lambda,|\Omega|)>0$,
which leads to
\begin{align*}
[\cal{M}_{\alpha/q}(\mathbbm{1}_{\Tilde{B}}Df)\omega]^q\in L^1(\Tilde{B}).
\end{align*}
Choosing a sequence $\{\eps_k\}$ such that $\eps_k\to0$ as $k\to\infty$, and letting $f_k=f_{\eps_k}$,
we finally apply the Lebesgue dominated convergence theorem to
deduce the conclusion.
\end{proof}

\begin{remark}
Recall $X=W^{1,p}_{\omega^p}(\Omega)$, $Y=W^{1,p}_{\omega^p}(\Omega)\cap \{v\in W^{1,1}(\Omega):|\M Dv|\in L^q(B)\}$ and \eqref{func-H}.
Letting $ m_Y=\displaystyle \inf_{v\in Y}\mathcal{H}(v,B)$
we see that
$$
m_Y\ge\inf_{u\in X}\mathcal{H}(u,B).
$$
Define the sequentially lower semicontinuous (s.l.s.c.) envelope
$$
\overline{\mathcal{H}}_Y=\sup\{\mathcal{G}:X\to[0,\infty]: 
\mbox{$\mathcal{G}$ is s.l.s.c., $\mathcal{G}\le\mathcal{H}$ on $Y$}\}.
$$
Since the constant functional $m_Y$ is s.l.s.c. and satisfies $m_Y\le \mathcal{H}$ on $Y$,
it follows $m_Y\le\overline{\mathcal{H}}_Y(u,B)$ for all $u\in X$ and hence, $m_Y=\inf_{u\in X}\overline{\mathcal{H}}_Y(u,B)$.
Moreover, in view of Lemma~\ref{lem:approx2}, for any $u\in X$ there is a sequence $\{u_k\}\subset W^{1,\infty}(B)$ such that
\begin{align}
\mbox{
$u_k\rightarrow u$ in $W^{1,p}_\omega(B)$ and
$\mathcal{H}(u_k,B)\rightarrow \mathcal{H}(u,B)$ 
as $k\rightarrow\infty$.}
\end{align}
We extend $\{u_k\}$ to whole domain so that $\{u_k\}\subset W^{1,\infty}(B)\cap Y$ under the assumption \eqref{303}.
We have from properties of the s.l.s.c. envelope $\overline{\mathcal{H}}_Y$ that
$$
m_Y
\le
\overline{\mathcal{H}}_Y(u,B)
\le
\liminf_{k\to\infty}\overline{\mathcal{H}}_Y(u_k,B)
\le
\liminf_{k\to\infty}\mathcal{H}(u_k,B)
\le
\mathcal{H}(u,B)
$$
for any $u\in X$.
Consequently,
$$\inf_{v\in Y}\mathcal{H}(v,B)=\inf_{u\in X}\mathcal{H}(u,B),$$
and thus the Lavrentiev phenomenon \eqref{302} is avoided.
For further discussion, see \cite{MR2076158,MR1351738}.
\end{remark}

\subsection{Admissible functions and distributional solutions}

In this subsection we briefly clarify admissible functions in the
context of distributions and solution spaces.

\begin{lemma}\label{lem:testing}
Let $B$ and $\tilde{B}$ be balls such that $B\Subset\Tilde{B}\Subset \Omega$.
Suppose that \eqref{holder-a}-\eqref{lavrentiev}, \eqref{G}, and \eqref{303}.
Let $f\in W^{1,1}(B)$ with $H(x,\M Df)\in L^1(B)$ be a distributional solution to the equation
\begin{align*}
  \dive\,\M A(x,\M Df)=\dive\,\M G(x,\M F)\quad\mbox{in $B$},  
\end{align*}
where $F:B\rightarrow\R^n$ is a given function with $H(x,\M F)\in
L^1(B)$, and $A:B\times\R^n\rightarrow\R^n$ is a Carath\'eodory map
with the growth condition
\begin{align}\label{a-growth-testing}
|A(x,z)|
\le
L(|z|^{p-1}+a(x)|z|^{q-1})
\end{align}
for a.e. $x\in B$, for all $z\in\R^n$, and for some constant $L>0$.
Then there holds
\begin{align}\label{testing-pde}
\int_B \<A(x,\M Df), \M D\varphi\>\,dx=
\int_B \<G(x,\M F), \M D\varphi\>\,dx.
\end{align}
for every $\varphi\in W^{1,1}_0(B)$ with $H(x,\M D\varphi)\in
L^1(B)$.
\end{lemma}
\begin{proof}
We may assume that $B=B_1(0)$ by dilation and translation.
For the result \eqref{testing-pde} it suffices to show that there
exists a sequence of functions $\{\varphi_k\}\subset C^\infty_0(B)$ such
that 
\begin{align}\label{seq-vaphi}
\mbox{$D\varphi_k\rightarrow D\varphi$ a.e. and $H(x,\M
D\varphi_k)\rightarrow H(x,\M D\varphi)$ in $L^1(B)$. }
\end{align}
Indeed, we have from \eqref{a-growth-testing}, \eqref{G} and Young's inequality
\begin{align*}
|\<A(x,\M Df), \M D\varphi_k\>|\le c\[H(x,\M Df)+H(x,\M D\varphi_k)\]
\end{align*}
and
\begin{align*}
|\<G(x,\M F), \M D\varphi_k\>|\le c\[H(x,\M F)+H(x,\M D\varphi_k)\]
\end{align*}
for some $c=c(L,p,q)>0$.
Using the Lebesgue dominated convergence theorem with \eqref{seq-vaphi}, this infers that
\begin{align*}
\<A(x,\M Df), \M D\varphi_k\>\rightarrow\<A(x,\M Df), \M D\varphi\>
\quad\mbox{strongly in $L^1(B)$}
\end{align*}
and
\begin{align*}
\<G(x,\M F), \M D\varphi_k\>\rightarrow\<G(x,\M F), \M D\varphi\>
\quad\mbox{strongly in $L^1(B)$}.
\end{align*}
Hence, \eqref{testing-pde} follows from
\begin{align*}
\int_B \<A(x,\M Df), \M D\varphi_k\>\,dx=
\int_B \<G(x,\M F), \M D\varphi_k\>\,dx.
\end{align*}

We now construct the sequence $\{\varphi_k\}$, modifying the arguments in the proof of Lemma~\ref{lem:approx2} and \cite[Proposition 3.1]{CM3}.
First, assume that 
$\varphi\in W^{1,1}_0(\R^n)$ and $H(x,\M D\varphi)\in L^1(\R^n)$  
with the zero extension of $\varphi$ to $\R^n\setminus B$.
The function $a$ can be extended to a H\"older continuous function on the whole $\R^n$, preserving the seminorm $[a]_{0,\alpha}$.
Consider a family $\{\phi_\eps\}$ of mollifiers as used in the proof Lemma~\ref{lem:approx2}.
For every $\eps\in(0,1/10)$ we define
$$\hat{\varphi}_\eps(x):=\varphi((1-3\eps)^{-1}x)
\quad\mbox{and}\quad
\hat{a}_\eps(x):=a((1-3\eps)^{-1}x)\quad \mbox{for $x\in\R^n$}.$$
Then $\supp(\hat{\varphi}_\eps)\subset B_{1-2\eps}(0)$.
We consider the mollified function $\varphi_\eps=(\hat{\varphi}_\eps\ast\phi_\eps)\in C_0^\infty(B_{1-\eps}(0))$ and
introduce the auxiliary functions
\begin{align*}
a_\eps(x)=\inf_{y\in B_\eps(x)} \hat{a}_\eps(y) \quad\mbox{and}\quad \quad
H_\eps(x,z)=|z|^p+a_\eps(x)|z|^q,
\quad\mbox{for $x\in B$ and $z\in\R^n$.}
\end{align*}
We then have from $|D\varphi_\eps-D\varphi|\le|(D\hat{\varphi}_\eps-D\varphi) \ast\phi_\eps|+|D\varphi\ast\phi_\eps-D\varphi|$
\begin{align*}
H(x,\M(x) D\varphi_\eps(x))\rightarrow H(x,\M(x) D\varphi(x)) \quad\mbox{a.e. in $B$, up to subsequence.}
\end{align*}
Also, by the H\"older continuity of $a$,
\begin{align}\label{xx-316}
H(x,\M(x) D\varphi_\eps(x))
&\le
H_\eps(x,\M(x) D\varphi_\eps(x))+c[a]_{0,\alpha}\eps^\alpha|\M(x)D\varphi_\eps(x)|^q
\quad\mbox{for $x\in B$.}
\end{align}
On the other hand, from the choice of $\phi_\eps$, $\hat{\varphi}_\eps$ and $a_\eps(x)$, we have that
\begin{align}\label{xx-317}
&H_\eps(x,\M(x) D\varphi_\eps(x))\nonumber\\
&\le
|\M(x)|^p\(\int_{B_\eps(x)}|D\hat{\varphi}_\eps(y)|\phi_\eps(x-y)\,dy\)^p
+
a_\eps(x)|\M(x)|^q
\(\int_{B_\eps(x)}|D\hat{\varphi}_\eps(y)|\phi_\eps(x-y)\,dy\)^q\nonumber\\
&\le
c\omega^p(x)\(\bint_{B_\eps(x)}|D\hat{\varphi}_\eps(y)|\,dy\)^p
+
c\omega^q(x)\(\bint_{B_\eps(x)}\hat{a}_\eps^{\frac{1}{q}}(y)|D\hat{\varphi}_\eps(y)|\,dy\)^q\nonumber\\
&\le
c\omega^p(x)\(\frac{1}{1-3\eps}\bint_{B_\eps(x)}|[D\varphi]((1-3\eps)^{-1}y)|\,dy\)^p\nonumber\\
&\quad\quad+
c\omega^q(x)\(\frac{1}{1-3\eps}\bint_{B_\eps(x)}a^{\frac{1}{q}}((1-3\eps)^{-1}y)|[D\varphi]((1-3\eps)^{-1}y)|\,dy\)^q
\end{align}
for some $c=c(n,p,q)>0$.
Setting $\tilde{B}_{\eps}(x):=B_{\eps/(1-3\eps)}(\frac{x}{1-3\eps})$, 
we observe 
$$
\left|\frac{x}{1-3\eps}-x\right|\le\frac{3\eps}{1-3\eps}
\quad\mbox{for all $x\in B=B_1(0)$,}
$$
which follows $\tilde{B}_{\eps}(x)\subset B_{\frac{4\eps}{1-3\eps}}(x)$.
Using this, maximal operators and $\eps\in(0,1/10)$ we obtain
\begin{align}\label{xx-318}
\omega^p(x)\(\frac{1}{1-3\eps}\bint_{B_\eps(x)}|[D\varphi]((1-3\eps)^{-1}y)|\,dy\)^p
&\le
2^p\omega^p(x)\(\bint_{\tilde{B}_{\eps}(x)}|D\varphi(s)|\,ds\)^p\nonumber\\
&\le
c\omega^p(x)\(\bint_{B_{\frac{4\eps}{1-3\eps}}(x)}|D\varphi(s)|\,ds\)^p\nonumber\\
&\le
c\cal{M}(|D\varphi|)^p(x)\omega^p(x)
\end{align}
for some $c=c(n,p)>0$.
Similarly, we deduce
\begin{align}\label{xx-319}
\omega^q(x)\(\frac{1}{1-3\eps}\bint_{B_\eps(x)}a^{\frac{1}{q}}((1-3\eps)^{-1}y)|[D\varphi]((1-3\eps)^{-1}y)|\,dy\)^q
\le
c\cal{M}(a^{\frac{1}{q}}|D\varphi|)^q(x)\omega^q(x)
\end{align}
for some $c=c(n,q)>0$.
We insert \eqref{xx-318} and \eqref{xx-319} into \eqref{xx-317} to get
\begin{align*}
H_\eps(x,\M(x) D\varphi_\eps(x))
\le
c\cal{M}(|D\varphi|)^p(x)\omega^p(x)
+
c\cal{M}(a^{\frac{1}{q}}|D\varphi|)^q(x)\omega^q(x)
\end{align*}
for $x\in B$, where $c=c(n,p,q)>0$.
We estimate the last term in \eqref{xx-316}.
Like the previous method we infer that
\begin{align*}
\eps^\alpha|\M(x)D\varphi_\eps(x)|^q
\le
c\omega^q(x)\bigg(\eps^{\alpha/q}\bint_{B_\eps(x)}|D\hat{\varphi}_{\eps}(y)|\,dy\bigg)^q
\le
c[\cal{M}_{\alpha/q}(D\varphi)(x)\omega(x)]^q.
\end{align*}
Choosing a sequence $\{\eps_k\}$ such that $\eps_k\to0$ as $k\to\infty$, we set $\varphi_k=\varphi_{\eps_k}$.
As the same way in the proof of Lemma~\ref{lem:approx2},
we apply Lemma~\ref{lem:Ap-prop}, Lemma~\ref{lem:Aps} and the Lebesgue dominated convergence theorem to complete the proof.
\end{proof}

\section{Properties of solutions to reference problems}\label{sec4}

In this section we present some properties of solutions to the
reference problems related to our problem \eqref{eq-u} which will be
employed later in proving our main results. To address an
appropriate function space of solutions to our reference problems,
we first introduce a related Musielak-Orlicz space. Let $B$ and
$\tilde{B}$ be balls such that $B\Subset\Tilde{B}\Subset \Omega$.
Under the assumptions \eqref{holder-a}-\eqref{lavrentiev} and that
\begin{align}\label{601}
|\log\M|_{BMO(\Tilde{B})} \le
W=\beta\min\left\{\frac{1}{s_o},\frac{1}{s}\right\},
\end{align}
where $s_o\in(p',\infty)$ as in Lemma~\ref{lem:Hsobolev}, $\frac{1}{s}=\frac{1}{p}-\frac{\alpha}{nq}$ and $\beta$ is
given as in Lemma~\ref{lem:w-prop}, we write
\begin{align*}
\begin{cases}
H_{\M}(x,z)=|\M(x)z|^p+a(x)|\M(x)z|^q \quad
\mbox{for $x\in\Omega$, $z\in\R^n$ or $\R$,}\\
\Tilde{H}_{\M}(x,t)=|t|^p\omega^p(x)+a(x)|t|^q\omega^q(x) \quad
\mbox{for $x\in\Omega$, $t\in\R$.}
\end{cases}
\end{align*}
Then we have $H_\M(x,z)=H(x,\M z)$ and
\begin{align*}
H_{\M}(x,z)\le\Tilde{H}_{\M}(x,|z|)\le \Lambda^qH_{\M}(x,z)
\end{align*}
from \eqref{cond-M}. We define the generalized Orlicz space
$L^{\Tilde{H}_\M}(B)$
\begin{align*}
L^{\Tilde{H}_\M}(B) = \left\{v\in L^1(B): \int_B
\Tilde{H}_\M(x,|v|)\,dx<\infty \right\}.
\end{align*}
This space is a separable reflexive Banach space
with the Luxemburg norm
\begin{align*}
\|v\|_{L^{\Tilde{H}_\M}(B)}
=
\inf\left\{
\varrho>0:
\int_B \Tilde{H}_\M\left(x,\frac{|v(x)|}{\varrho}\right)\,dx\le1
\right\},
\end{align*}
and we have that $L^q_{\omega^q}(B)\subset
L^{\Tilde{H}_\M}(B)\subset L^p_{\omega^p}(B)\subset L^1(B)$.
Thanks to
\begin{align*}
\min\{\lambda^p\Tilde{H}_{\M}(x, t),\lambda^q\Tilde{H}_{\M}(x,t)\}
\le
\Tilde{H}_{\M}(x,\lambda t)
\le
\max\{\lambda^p\Tilde{H}_{\M}(x, t),\lambda^q\Tilde{H}_{\M}(x,t)\},
\end{align*}
for any $\lambda>0$,
we infer from the definition of Luxemburg norm that
\begin{align}\label{norm-H1}
\min\left\{\|v\|_{L^{\Tilde{H}_{\M}}(B)}^p,\|v\|_{L^{\Tilde{H}_{\M}}(B)}^q\right\}
\le
\int_B\Tilde{H}_{\M}(x,v)\,dx
\le
\max\left\{\|v\|_{L^{\Tilde{H}_{\M}}(B)}^p,\|v\|_{L^{\Tilde{H}_{\M}}(B)}^q\right\}.
\end{align}
 And the Musielak-Orlicz-Sobolev space $W^{1,\Tilde{H}_\M}(B)$ consists of
all functions $v\in L^{\Tilde{H}_\M}(B)\cap W^{1,1}(B)$ such that
its weak partial derivatives $D_{x_i}v$ belong to
$L^{\Tilde{H}_\M}(B)$, and its norm is defined by
\begin{align*}
\norm{v}_{W^{1,\Tilde{H}_\M}(B)}
=
\norm{v}_{L^{\Tilde{H}_\M}(B)}
+
\norm{|Dv|}_{L^{\Tilde{H}_\M}(B)}.
\end{align*}
Moreover, $W^{1,\Tilde{H}_\M}_0(B)$ is defined as the closure of
$C_0^\infty(B)$ in $W^{1,\Tilde{H}_\M}(B)$. This definition for the
zero boundary Sobolev space related to $\Tilde{H}_\M$ is reasonable,
as we see that smooth functions are dense in $W^{1,\Tilde{H}_\M}(B)$
like the way in the proof of Lemma~\ref{lem:approx2} from $p'<s_o$ and $p<q\le s$. 
In view of Lemma~\ref{lem:w-prop}, we also note that
the assumptions \eqref{holder-a}-\eqref{lavrentiev}
and \eqref{601} guarantee the Sobolev Poincar\'e type inequality in
Lemma~\ref{lem:Hsobolev}. For a further discussion on
Musielak-Orlicz spaces and their associated Sobolev spaces, we refer
to \cite{MR724434,HH,MR3889985,MR3298478, MR3516828} and the
references therein.

\

Throughout this section we basically assume the followings. 
Let $B$
and $\Tilde{B}$ be balls such that $B=B_\rho\Subset\Tilde{B}\Subset
\Omega$ and $\rho\in(0,1]$, and assume \eqref{holder-a}-\eqref{lavrentiev}, \eqref{601} and \eqref{ApAq}.
We introduce the auxiliary functions
\begin{align*}
\ov{A}_p(z)=\bint_{B}A_p(x,z)\,dx, \quad
\ov{A}_q(z)=\bint_{B}A_q(x,z)\,dx,
\end{align*}
\begin{align*}
V_p(z) =|z|^{\frac{p-2}{2}}z, \quad\mbox{and}\quad
V_q(z)=|z|^{\frac{q-2}{2}}z
\end{align*}
for $z\in\R^n\setminus\{0\}$.
Then the assumption \eqref{ApAq} implies
\begin{align}\label{v-v}
\begin{cases}
|V_p(z_1)-V_p(z_2)|^2
\le
c\<A_p(x,z_1)-A_p(x,z_2), z_1-z_2\>,\\
|V_q(z_1)-V_q(z_2)|^2
\le
c\<A_q(x,z_1)-A_q(x,z_2), z_1-z_2\>
\end{cases}
\end{align}
for some constant $c=c(n,\nu,p,q)>0$. Substituting $\ov{A}_p$ for
$A_p$ and $\ov{A}_q$ for $A_q$, respectively, still works with
\eqref{ApAq}, and hence with \eqref{v-v} for $\ov{A}_p$ and
$\ov{A}_q$ as well. Next we use the notations
\begin{align*}
\ov{\M}=\(\M\)_B^{\log} \quad\mbox{and}\quad
\ov{\omega}=\(\omega\)_B^{\log}.
\end{align*}
We also use the symbol
\begin{align*}
\data_0\equiv\data_0(n,p,q,\alpha,\nu,L,\Lambda, [a]_{0,\alpha}),
\end{align*}
to represent variable constants that depend on $n,p,q,\alpha,\nu,L,\Lambda$, and $[a]_{0,\alpha}$.

\subsection{Reference problem 1}

With \eqref{G} we first consider the Dirichlet problem
\begin{align}\label{701}
\begin{cases}
\dive\( \M A_p(x,\M Dh)+a(x)\M A_q(x,\M Dh) \)=\dive\(\M G(x,\M F_o)\)
\quad\mbox{in $B$},\\
h= h_0 \quad\mbox{on $\partial B$},
\end{cases}
\end{align}
where $h_0\in W^{1,1}(B)$ with $H(x,\M Dh_0)\in L^1(B)$ and
$F_o:B\rightarrow\R^n$ with $H(x,\M F_o)\in L^1(B)$. 
Then the following lemma guarantees the existence of a unique weak solution $h$ with $h-h_0\in W^{1,\Tilde{H}_\M}_0(B)$, 
based on the monotonicity method in Musielak-Orlicz spaces, 
see \cite{MR4361836,S1,MR3071548,BOh1,MR3516828} for further discussion.

\begin{lemma}
There exists a unique weak solution to \eqref{701}.
\end{lemma}
\begin{proof}
Keeping in mind the condition \eqref{601}, we first see that 
$W^{1,\Tilde{H}_\M}_0(B)$ is a separable reflexive Banach space
and its norm is equivalently defined by
$\|Dv\|_{L^{\Tilde{H}_{\M}}(B)}$ for $v\in W^{1,\Tilde{H}_\M}_0(B)$ in view of Lemma~\ref{lem:Hsobolev}.
Letting $v=h-h_o\in W^{1,\Tilde{H}_\M}_0(B)$, \eqref{701} is rewritten by
\begin{align}\label{eq-v-XX}
\begin{cases}
\dive\( \M A_p(x,\M (Dv+Dh_0))+a(x)\M A_q(x,\M (Dv+Dh_0)) \)\hspace{-0.03cm}
=\hspace{-0.03cm}
\dive\(\M G(x,\M F_o)\)
\mbox{ in $B$},\\
v= 0 \quad\mbox{on $\partial B$}.
\end{cases}
\end{align}
Now, define an operator $T:W^{1,\Tilde{H}_\M}_0(B)\rightarrow \(W^{1,\Tilde{H}_\M}_0(B)\)'$ by
\begin{align}
[T(v)](\varphi)=
\int_B\<A_p(x,\M (Dv+Dh_0))+a(x)\M A_q(x,\M (Dv+Dh_0)), \M D\varphi\>\,dx.
\end{align}
Then $T$ is continuous and monotone from \eqref{ApAq}.
We now show that $T$ is coercive, i.e.,
\begin{align}
\frac{[T(v)](v)}{\|v\|}\to\infty
\quad\mbox{as $\|v\|\to\infty$,}
\end{align}
where $\|\cdot\|=\|\cdot\|_{W^{1,\Tilde{H}_\M}_0(B)}$.
To show this, we use \eqref{ApAq} and Young's inequality to deduce
\begin{align}
[T(v)](v)
&=
\int_B\<A_p(x,\M (Dv+Dh_0))+a(x)\M A_q(x,\M (Dv+Dh_0)), \M Dv\>\,dx\nonumber\\
&=
\int_B\<A_p(x,\M (Dv+Dh_0))+a(x)\M A_q(x,\M (Dv+Dh_0)), \M (Dv+Dh_0)\>\,dx\nonumber\\
&\quad\quad-
\int_B\<A_p(x,\M (Dv+Dh_0))+a(x)\M A_q(x,\M (Dv+Dh_0)), \M Dh_0\>\,dx
\nonumber\\
&\ge
c\int_B H(x,\M Dv)\,dx-c\int_BH(x,\M Dh_0)\,dx
\end{align}
for some $c=c(p,q,L,\nu)>0$.
According to \eqref{norm-H1} this leads to
\begin{align}
[T(v)](v)
\ge
c\min\left\{\|Dv\|_{L^{\Tilde{H}_{\M}}(B)}^p,\|Dv\|_{L^{\Tilde{H}_{\M}}(B)}^q\right\}
-
c\|H(\cdot,\M Dh_0)\|_{L^1(B)}
\end{align}
This follows from $1<p<q$
\begin{align}
\frac{[T(v)](v)}{\|v\|}
&\ge
\frac{c\min\left\{\|Dv\|_{L^{\Tilde{H}_{\M}}(B)}^p,\|Dv\|_{L^{\Tilde{H}_{\M}}(B)}^q\right\}
-
c\|H(\cdot,\M Dh_0)\|_{L^1(B)}
}{\|Dv\|_{L^{\Tilde{H}_{\M}}(B)}}
\to\infty
\end{align}
as $\|Dv\|_{L^{\Tilde{H}_{\M}}(B)}\to\infty$.
Emplying Minty-Browder Theorem (\cite[Theorem 2.2, Ch. II]{S1}) and considering
$[F_o]\in \(W^{1,\Tilde{H}_\M}_0(B)\)'$ defined by
\begin{align}
[F_o](\varphi)=\int_B\< G(x,\M F_o), \M D\varphi\>\,dx,
\end{align}
there is a weak solution $v\in W^{1,\Tilde{H}_\M}_0(B)$ to \eqref{eq-v-XX}.
Therefore, we obtain a weak solution $h=v+h_0$ to \eqref{701}.

We next show that a weak solution to \eqref{701} is unique. 
Let $h_1, h_2\in h_0+W^{1,\Tilde{H}_\M}_0(B)$ be weak solutions to \eqref{701}.
Testing $h_1-h_2$ to \eqref{701}, we have
\begin{align*}
\int_B\<A_p(x,\M Dh_1)+a(x)\M A_q(x,\M Dh_1), \M( Dh_1- Dh_2)\>\,dx
=
\int_B\< G(x,\M F_o), \M( Dh_1- Dh_2)\>\,dx
\end{align*}
and
\begin{align*}
\int_B\<A_p(x,\M Dh_2)+a(x)\M A_q(x,\M Dh_2), \M( Dh_1- Dh_2)\>\,dx
=
\int_B\< G(x,\M F_o), \M( Dh_1- Dh_2)\>\,dx.
\end{align*}
This follows from \eqref{v-v}
\begin{align*}
0
&=
\int_B\<A_p(x,\M Dh_1)-\M A_p(x,\M Dh_2), \M( Dh_1- Dh_2)\>\,dx\nonumber\\
&\quad\quad+
\int_Ba(x)\<A_q(x,\M Dh_1)-\M A_q(x,\M Dh_2), \M( Dh_1- Dh_2)\>\,dx\nonumber\\
&\ge
c\int_B |V_p(\M Dh_1)-V_p(\M Dh_2)|^2+a(x) |V_q(\M Dh_1)-V_q(\M Dh_2)|^2\,dx
\end{align*}
for some $c=c(n,\nu,p,q)>0$, 
and hence $h_1=h_2$.
\end{proof}

\begin{lemma}\label{lem:h-energy}
Let $h$ be the weak solution to \eqref{701}.
Then we have the energy estimate
\begin{align*}
\int_{B}H(x,\M Dh)\,dx
\le
c\int_{B}H(x,\M Dh_0)\,dx
+
c\int_{B}H(x,\M F_o)\,dx
\end{align*}
for some $c=c(n,p,q,\nu,L)>0$.
\end{lemma}
\begin{proof}
We test the weak formulation of \eqref{701} with $h-h_0$ to get
\begin{align*}
\int_{B}
\<A_p(x,\M Dh)+
a(x)A_q(x,\M Dh), \M Dh\>\,dx
&=
\int_{B}\<A_p(x,\M Dh)+
a(x)A_q(x,\M Dh), \M Dh_0\>\,dx\nonumber\\
&\quad\quad+
\int_{B} \< G(x,\M F_o),\M Dh-\M Dh_0\>\,dx.
\end{align*}
Using \eqref{ApAq}, \eqref{G} and Young's inequality,
we have
\begin{align*}
\int_{B}H(x,\M Dh)\,dx
&\le
c\int_{B}\big(|\M Dh|^{p-1}+a(x)|\M Dh|^{q-1}\big)|\M Dh_0|\,dx\nonumber\\
&\quad\quad+
c\int_{B}\big(|\M F_o|^{p-1}+a(x)|\M F_o|^{q-1}\big)(|\M Dh|+|\M Dh_0|)\,dx\nonumber\\
&\le
\tfrac{1}{2}\int_{B}H(x,\M Dh)\,dx
+
c\int_{B}H(x,\M Dh_0)\,dx
+
c\int_{B}H(x,\M F_o)\,dx
\end{align*}
for some $c=c(n,p,q,\nu,L)>0$. We then have the conclusion.
\end{proof}

\begin{lemma}\label{lem:h-higher}
Let $h$ be the weak solution to \eqref{701}. Then the followings
hold true.
\begin{enumerate}
\item 
If $H(x,\M F_o)\in L^{1+\hat{\sigma}}(B)$ for some $\hat{\sigma}>0$, then there
exists $\sigma_\ast\in(0,\hat{\sigma})$, depending only on $\data_0$, $\hat{\sigma}$
and $\norm{H(x,\M Dh)}_{L^1(B)}$,
such that we have $H(x,\M Dh)\in L^{1+\sigma_*}_{\loc}(B)$ and
\begin{align*}
\(\bint_{\frac{1}{2}B}\[ H(x,\M Dh)\]^{1+\sigma}\,dx\)^{\frac{1}{1+\sigma}}
\le
c\bint_{B} H(x,\M Dh)\,dx
+
\(\bint_{B}\[ H(x,\M F_o)\]^{1+\sigma}\,dx\)^{\frac{1}{1+\sigma}}
\end{align*}
for any $\sigma\in(0,\sigma_\ast]$,
where $c=c(\data_0, \sigma,\norm{H(x,\M Dh)}_{L^1(B)})>0$.

\item If $F_o\equiv0$ and
$[H(x,\M Dh_0)]^{1+\sigma_o}\in L^1(B)$ for some $\sigma_o>0$,
then there exists $\sigma_{**}\in(0,\sigma_o)$, depending only on
$\data_0,\sigma_o,$
and $\norm{H(\cdot,\M Dh_0)}_{L^1(B)}$, such that
we have $H(x,\M Dh)\in L^{1+\sigma_{**}}(B)$ and
\begin{align*}
\int_{B}[H(x,\M Dh)]^{1+\sigma}\,dx
\le
c\int_{B}[H(x,\M Dh_0)]^{1+\sigma}\,dx
\end{align*}
for any $\sigma\in(0,\sigma_{**}]$,
where $c=c(\data_0,\sigma, \sigma_o,\norm{H(x,\M Dh_0)}_{L^1(B)})>0$.
\end{enumerate}
\end{lemma}
\begin{proof}
(1) Let $y\in B$ and $0<2r<dist(y,\partial B)$ so that
$B_{2r}(y)\subset B$, and choose a cut-off function $\eta\in
C^\infty_0(B_{2r}(y))$ such that
$\mathbbm{1}_{B_r(y)}\le\eta\le\mathbbm{1}_{B_{2r}(y)}$ and
$|D\eta|\le\frac{2}{r}$. We then take $\(h-(h)_{B_{2r}(y)}\)\eta^q$
as a test function in the weak formulation of \eqref{701}. 
By means of \eqref{ApAq}, \eqref{G} and Young's inequality, we obtain
\begin{align*}
\int_{B_{2r}(y)}\eta^qH(x,\M Dh)\,dx
&\le
c\int_{B_{2r}(y)}\eta^{q-1}\(|\M Dh|^{p-1}+a(x)|\M Dh|^{q-1}\)\frac{\abs{\M\(h-(h)_{B_{2r}(y)}\)}}{r}\,dx
\nonumber\\
&\quad +
c\int_{B_{2r}(y)}\eta^{q-1}\(|\M F_o|^{p-1}+a(x)|\M F_o|^{q-1}\)\frac{\abs{\M\(h-(h)_{B_{2r}(y)}\)}}{r}\,dx
\nonumber\\
&\quad + c\int_{B_{2r}(y)}\eta^{q}\(|\M F_o|^{p-1}+a(x)|\M F_o|^{q-1}\) |\M Dh|\,dx
\nonumber\\
&\le \tfrac{1}{2} \int_{B_{2r}(y)}\eta^qH(x,\M Dh)\,dx
+ c\int_{B_{2r}(y)}H\bigg(x,\frac{\M\(h-(h)_{B_{2r}(y)}\)}{r}\bigg)\,dx
\nonumber\\
&\quad + c\int_{B_{2r}(y)}H(x,\M F_o)\,dx
\end{align*}
for some $c=c(n,p,q,\nu,L)>0$. Thus
\begin{align*}
\int_{B_{2r}(y)}\eta^qH(x,\M Dh)\,dx
&\le
c\int_{B_{2r}(y)}H\bigg(x,\frac{\M\(h-(h)_{B_{2r}(y)}\)}{r}\bigg)\,dx
+ c\int_{B_{2r}(y)}H(x,\M F_o)\,dx.
\end{align*}
We then apply the Sobolev-Poincar\'e type inequality in Lemma~\ref{lem:Hsobolev},
to obtain the reverse H\"older inequality
\begin{align}\label{int-reverse-hol}
\bint_{B_r(y)} H(x,\M Dh)\,dx
&\le c\(\bint_{B_{2r}(y)} \[H(x,\M Dh)\]^{\theta_o}\,dx\)^{\frac{1}{\theta_o}}
+ c\bint_{B_{2r}(y)}H(x,\M F_o)\,dx,
\end{align}
for any $B_{2r}(y)\subset B$, where the positive constant $c$ depends on $\data_0$, and $\norm{H(\cdot,\M Dh)}_{L^1(B)}$ and $\theta_o\in(0,1)$ is given as in Lemma~\ref{lem:Hsobolev}.
Therefore, using Gehring's lemma,
we arrive at the desired conclusion.

\

(2) We first note from Lemma~\ref{lem:h-energy} that
\begin{align}\label{710}
\norm{H(x,\M Dh)}_{L^1(B)}\le c\norm{H(x,\M Dh_0)}_{L^1(B)}
\end{align}
for some constant $c=c(n,p,q,\nu,L)>0$.

For any $y\in B$, we consider a ball $B_{2r}(y)$ such that
$|B_{2r}(y)\setminus B|>|B_{2r}(y)|/10$. Now, take
$\varphi=(h-h_0)\eta^q$ as a test function to \eqref{701}, where
$\eta\in C^\infty_0(B_{2r}(y))$ with
$\mathbbm{1}_{B_r(y)}\le\eta\le\mathbbm{1}_{B_{2r}(y)}$ and
$|D\eta|\le\frac{2}{r}$. Note that $\varphi$ is admissible in the
weak formulation. As in the proof of (1), we deduce that
\begin{align}\label{711}
&\int_{B_{2r}(y)\cap B}\eta^qH(x,\M Dh)\,dx\nonumber\\
&\le
c\int_{B_{2r}(y)\cap B}\eta^qH(x,\M Dh_0)\,dx
+
c\int_{B_{2r}(y)\cap B}H\bigg(x,\frac{\M\(h-h_0\)}{r}\bigg)\,dx,
\end{align}
where $c=c(n,p,q,\nu,L)>0$.
We then apply Lemma~\ref{lem:Hsobolev} with \eqref{710} in the last integral on the right-hand side in \eqref{711}, to yield
\begin{align*}
&\frac{1}{|B_{2r}(y)|}\int_{B_{2r}(y)\cap B}H\bigg(x,\frac{\M\(h-h_0\)}{r}\bigg)\,dx\\
&\le
c\(\frac{1}{|B_{2r}(y)|}\int_{B_{2r}(y)\cap B}[H(x,\M (Dh-Dh_0))]^{\theta_o}\,dx\)^{\frac{1}{\theta_o}}\\
&\le
c\(\frac{1}{|B_{2r}(y)|}\int_{B_{2r}(y)\cap B}[H(x,\M Dh)]^{\theta_o}\,dx\)^{\frac{1}{\theta_o}}
+
\frac{c}{|B_{2r}(y)|}\int_{B_{2r}(y)\cap B}H(x,\M Dh_0)\,dx
\end{align*}
for some $c=c(\data_0,\norm{H(\cdot,\M Dh_0)}_{L^1(B)})>0$,
where $\theta_o\in(0,1)$ is given as in Lemma~\ref{lem:Hsobolev}. 
This and \eqref{711} imply
\begin{align*}
\bint_{B_{r}(y)}H(x,\M Dh)\,dx
&\le
c\(\bint_{B_{2r}(y)}[H(x, \M Dh)]^{\theta_o}\,dx\)^{\frac{1}{\theta_o}}
+
c\bint_{B_{2r}(y)}H(x,\M Dh_0)\,dx
\end{align*}
with $H(x,\M Dh)=0$ and $H(x,\M Dh_0)=0$ in $B_{2r}(y)\setminus B$,
where the positive constant $c$ depends only on $\data_0$ and $\norm{H(\cdot,\M Dh_0)}_{L^1(B)}$.
The previous inequality holds for any $B_{2r}(y)\subset B$ from \eqref{int-reverse-hol} with $F_o\equiv0$ in the proof of (1).
Therefore, we conclude the result by Gehring's lemma.
\end{proof}

\begin{remark}
\label{rem:h}
It is possible to apply Lemmas~\ref{lem:h-energy} and \ref{lem:h-higher} to the problem with $\ov{\M}$ instead of $\M$ in \eqref{701}.
\end{remark}

In view of Lemma~\ref{lem:h-higher},
we deduce the following Proposition.
\begin{proposition}\label{prop:u-higher}
Under the assumptions in Theorem~\ref{thm}, there exists $\tilde{\sigma}\in(0,\gamma-1)$, depending only on $\data$, such that
\begin{align}
H(x,\M Du)\in L^{1+\tilde{\sigma}}_{\loc}(\Omega)
\end{align}
and
\begin{align}
\bint_{B_{\rho/2}} H(x,\M Du)^{1+\sigma}\,dx
\le
c\bigg(\bint_{B_{\rho}} H(x,\M Du)\,dx\bigg)^{1+\sigma}
+
c\bint_{B_{\rho}} H(x,\M F)^{1+\sigma}\,dx
\end{align}
for any $\sigma\in(0,\tilde{\sigma}]$ and for any ball $B_\rho\subset\Omega$ with $0<\rho\le1$,
where $c=c(\data,\sigma)>0$.
Moreover, by a standard covering argument,
\begin{align}
\|H(\cdot,\M Du)\|_{L^{1+\tilde{\sigma}}(\Omega_{2\hat{\tau}})}
\le
c(\data,\hat{\tau},\gamma,\|H(\cdot,\M F)\|_{L^{\gamma}(\Omega_{\hat{\tau}})}).
\end{align}
\end{proposition}

We next provide a useful lemma which can be applied to functions with the higher integrability property of Lemma~\ref{lem:h-higher}.
\begin{lemma}\label{cor:h}
Let $f\in W^{1,1}(B)$ with $H(x,\M Df)\in L^{1+\sigma_*}(B)$. Then
for any $\sigma\in(0,\sigma_*)$, there is
$\delta=\delta(n,p,q,\sigma,\sigma_*)>0$ such that if $|\log
\M|_{\BMO} \leq \delta$, then
 we have $H(x,\ov{\M}Df)\in L^{1+\sigma}(B)$ with the estimate
 \begin{align}\label{716}
\(\bint_{B}\[H(x,\ov{\M} Df)\]^{1+\sigma}\,dx\)^{\frac{1}{1+\sigma}}
\le
c\(\bint_{B}\[H(x,\M Df)\]^{1+\sigma_*}\,dx\)^{\frac{1}{1+\sigma_*}},
\end{align}
where $c=c(n,p,q,\sigma,\sigma_*)>0$.
\end{lemma}
\begin{proof}
For $\sigma\in(0,\sigma_*)$, we have that
\begin{align*}
&\bint_{B}\[H(x,\ov{\M} Df)\]^{1+\sigma}\,dx\nonumber\\
&\le
c\bint_{B}\(\frac{|\ov{\M}|}{|\M|}\)^{p(1+\sigma)}|\M Df|^{p(1+\sigma)}\,dx
+c\bint_{B}\(\frac{|\ov{\M}|}{|\M|}\)^{q(1+\sigma)}\(a(x)|\M Df|^q\)^{1+\sigma}\,dx\nonumber\\
&\le
c\(\bint_{B}\[H(x,\M Df)\]^{1+\sigma_*}\,dx\)^{\frac{1+\sigma}{1+\sigma_*}}
\[\bint_{B}
\(\frac{|\ov{\M}|}{|\M|}\)^{\frac{p(1+\sigma)(1+\sigma_*)}{\sigma_*-\sigma}}
+
\(\frac{|\ov{\M}|}{|\M|}\)^{\frac{q(1+\sigma)(1+\sigma_*)}{\sigma_*-\sigma}}
\,dx\]^{\frac{\sigma_*-\sigma}{1+\sigma_*}}
\end{align*}
for some $c=c(n,p,q,\sigma,\sigma_*)>0$.
Now, Lemma~\ref{lem:w-prop} and \eqref{ovM} imply
\begin{align*}
\bint_{B}\[H(x,\ov{\M} Df)\]^{1+\sigma}\,dx
\le
c\(\bint_{B}\[H(x,\M Df)\]^{1+\sigma_*}\,dx\)^{\frac{1+\sigma}{1+\sigma_*}}
\end{align*}
by selecting $\delta=\delta(n,p,q,\sigma,\sigma_*)>0$ sufficiently small.
\end{proof}

\begin{remark}
\label{rem:h2} 
In Lemma~\ref{cor:h} one can switch $\M$ and $\ov{\M}$ under small log-BMO control of $\M$. 
In other words, under the assumption
$f\in W^{1,1}(B)$ with $H(x,\ov{\M}Df)\in L^{1+\sigma_*}(B)$, there
is $\delta=\delta(n,p,q,\sigma,\sigma_*)>0$ for any
$\sigma\in(0,\sigma_*)$ such that if $|\log \M|_{\BMO} \leq \delta$,
then we have $H(x,\M Df)\in L^{1+\sigma}(B)$ with the analogous
estimate of \eqref{716}.
\end{remark}

\subsection{Reference problem 2}
We next consider the problem
\begin{align}\label{801}
\begin{cases}
\dive \( \ov{\M} \ov{A}_p(\ov{\M} Dw)+a(x)\ov{\M} \ov{A}_q(\ov{\M} Dw) \)
=
0
\quad\mbox{in $B$,}\\
w=w_0\quad\mbox{on $\partial B$,}
\end{cases}
\end{align}
where $w_0\in W^{1,1}(B)$ is a given boundary datum such that
\begin{align}\label{802}
|\ov{\M}Dw_0|^p+a(x)|\ov{\M}Dw_0|^q\in L^{1+\sigma_0}(B) \quad
\mbox{for some $\sigma_0>0$},
\end{align}
and
\begin{align}\label{803}
\norm{H(x,\ov{\M}Dw_0)}_{L^1(B)} \le C_0
\end{align}
for some constant $C_0>0$.

Defining the auxiliary map $\Phi:\Omega\times\R^n\rightarrow\R^n$ such that
\begin{align*}
\Phi(x,\eta):=\frac{\ov{\M}\ov{A}_p(\ov{\M}\eta)}{|\ov{\M}|^p}
+
\tilde{a}(x)\frac{\ov{\M}\ov{A}_q(\ov{\M}\eta)}{|\ov{\M}|^q},
\end{align*}
where $\tilde{a}(x):=a(x)|\overline{\M}|^{q-p}$,
we have that $w$ solves
\begin{align}\label{805}
\begin{cases}
\dive\,\Phi(x,Dw)=0
\quad\mbox{in $B$},\\
w=w_0
\quad\mbox{on $\partial B$}
\end{cases}
\end{align}
with
$|Dw_0|^p+\tilde{a}(x)|Dw_0|^q\in L^{1+\sigma_0}(B)$
and
\begin{align*}
\int_B|Dw_0|^p+\tilde{a}(x)|Dw_0|^q\,dx \le |\ov{\M}|^{-p}\Lambda^q
C_0.
\end{align*}
We observe from \eqref{cond-M} and \eqref{ApAq}
\begin{align*}
\begin{cases}
|\Phi(x,\eta)|+|D_\eta\Phi(x,\eta)||\eta|
\le
L(|\eta|^{p-1}+\tilde{a}(x)|\eta|^{q-1}),\\
\<D\Phi(\eta)\zeta,\zeta\>\ge\tilde{\nu}(|\eta|^{p-2}+\tilde{a}(x)|\eta|^{q-2})|\zeta|^2,\\
|\Phi(x_1,\eta)-\Phi(x_2,\eta)|
\le
L(\tilde{a}(x_1)-\tilde{a}(x_2))|\eta|^{q-1},
\end{cases}
\end{align*}
where $\tilde{\nu}=\frac{\nu}{\Lambda^{q+1}}$. 
We point out 
that the weak solution $w$ to \eqref{805} is already known to possess some higher
integrability, as established by the results in Sections 3 and 4 of \cite{MR3985927}. 
However, we here want higher integrability
estimates of the gradient $Dw$ coupled with the constant matrix $\ov{\M}$,
involving constants $c$ independent of the quantity $|\ov{\M}|$ in
the estimates. 

\

We first state the analogs of Lemma~\ref{lem:h-energy} and
Lemma~\ref{lem:h-higher} for the weak solution $w$ of \eqref{801}.
Like the proof of Lemma~\ref{lem:h-energy}, we take a test function
$\varphi =w-w_0$ in the weak formulation of \eqref{801} to get
\begin{align}\label{436-d}
\int_B H(x,\ov{\M}Dw)\,dx
\le
c\int_B H(x,\ov{\M}Dw_0)\,dx
\end{align}
for some $c=c(n,p,q,\nu,L)>0$.
In addition, there exists $\sigma_\ast\in(0,\sigma_0)$, depending only on $\data_0,C_0$ and $\sigma_0$, such that
\begin{align}\label{809-a}
\(\bint_{\frac{1}{2} B}\[ H(x,\ov{\M} Dw)\]^{1+\sigma}\,dx\)^{\frac{1}{1+\sigma}}
\le
c\bint_{B} H(x,\ov{\M} Dw)\,dx
\end{align}
and
\begin{align}\label{810-a}
\int_{B}[H(x,\ov{\M} Dw)]^{1+\sigma}\,dx
\le
c\int_{B}[H(x,\ov{\M} Dw_0)]^{1+\sigma}\,dx
\end{align}
for any $\sigma\in(0,\sigma_{*}]$, where $c=c(\data_0,C_0, \sigma_0,
\sigma)>0$. The proof is similar to the proof of
Lemma~\ref{lem:h-higher}.

\begin{lemma}[\cite{MR3985927}]\label{lem:w-reg}
Let $w$ be the weak solution to \eqref{801} under \eqref{802} and \eqref{803}.
Then we have
\begin{align*}
Dw\in L^{\frac{np}{n-2\beta}}_{\loc}(B)\cap W_{\loc}^{\min\{2\beta/p,\beta\},p}(B)
\end{align*}
for every $\beta<\alpha$. In particular, we have
\begin{align*}
Dw\in L^{2q-p}_\loc(B).
\end{align*}
\end{lemma}

Based on Lemma~\ref{lem:w-reg}, we deduce the following lemma.
\begin{lemma}\label{lem:w-higher}
Let $B=B_{4r}$ with $r\in(0,\frac{1}{4}]$ and let $w$ be the weak solution to \eqref{801} under \eqref{802} and \eqref{803}.
Suppose that
\begin{align}\label{808}
\sup_{x\in B_r} a(x)\le K[a]_{0,\alpha}r^\alpha
\end{align}
for some $K\ge1$.
Then, for any $\bar{q}<np/(n-2\alpha)$($=\infty$ when $\alpha=1$ and $n=2$)
we have
\begin{align}\label{809}
\(\bint_{B_r}|\barM Dw|^{\bar{q}}\,dx\)^{1/\bar{q}}
\le
c\(\bint_{B_{4r}}H(x,\barM Dw)\,dx\)^{1/p},
\end{align}
where $c=c(\data,\sigma_0,C_0,\bar{q},K)>0$.
Furthermore, for $1<\gamma_1<\gamma_2\le4$, we have
\begin{align}\label{810}
\(\bint_{B_{\gamma_1r}}|\barM Dw|^{2q-p}\,dx\)^{1/(2q-p)}
\le
c\(\bint_{B_{\gamma_2r}}H(x,\barM Dw)\,dx\)^{1/p}
\end{align}
for some $c=c(\data,\sigma_0,C_0,K,\gamma_1, \gamma_2)>0$.
\end{lemma}
\begin{proof}
We note that 
$\ov{\M}Dw\in L^{\frac{np}{n-2\beta}}_{\loc}(B)\cap
W_{\loc}^{\min\{2\beta/p,\beta\},p}(B)$
for any $\beta<\alpha$, in
particular, 
$\ov{\M}Dw\in L^{2q-p}_\loc(B)$
from Lemma~\ref{lem:w-reg}.
This means that we do not require the approximation arguments like the proof of \cite[Theorem 5.1]{CM1}  
Using rescaled functions, we may assume that $B_1\Subset B$.
Let $\eta\in C_0^\infty(B_{3/4})$ be a cut-off function such that
$0\le\eta\le1$, $\eta=1$ in $B_{2/3}$ and $|D\eta|^2+|D^2\eta|\le
10^4$. Choose $h\in\R^n$ so that $0<|h|\le10^{-4}$. Defining the
finite difference operator
$$
\tau_hf(x)\equiv \tau_h(f)(x)=f(x+h)-f(x)
$$
for a function $f$,
we take a test function $\varphi=\tau_{-h}(\eta^2\tau_h w)$ in \eqref{801} to get
\begin{align*}
I_0
&=
\int_{B_{1}}\eta^2(x)
\<\ov{A}_p(\ov{\M}Dw(x+h))-\ov{A}_p(\ov{\M}Dw(x)),\ \tau_h(\ov{\M} Dw)(x)\>\,dx\nonumber\\
&\quad\quad+
\int_{B_{1}}\eta^2(x)
a(x+h)\<\ov{A}_q(\ov{\M}Dw(x+h))-\ov{A}_q(\ov{\M}Dw(x)),\ \tau_h(\ov{\M} Dw)(x)\>\,dx\nonumber\\
&=
-
\int_{B_{1}}\eta^2(x)
(a(x+h)-a(x))\<\ov{A}_q(\ov{\M}Dw(x)),\ \tau_h(\ov{\M} Dw)(x)\>\,dx\nonumber\\
&\quad\quad-
2\int_{B_{1}}\eta(x)
\tau_h(\ov{\M} w)(x)
(a(x+h)-a(x))\<\ov{A}_q(\ov{\M}Dw(x)),\ D\eta(x)\> \,dx\nonumber\\
&\quad\quad-
2\int_{B_{1}}\eta(x)
\tau_h(\ov{\M} w)(x)
\<\ov{A}_p(\ov{\M}Dw(x+h))-\ov{A}_p(\ov{\M}Dw(x)),\ D\eta(x)\> \,dx\nonumber\\
&\quad\quad-
2\int_{B_{1}}\eta(x)
a(x+h)
\tau_h(\ov{\M} w)(x)
\<\ov{A}_q(\ov{\M}Dw(x+h))-\ov{A}_q(\ov{\M}Dw(x)),\ D\eta(x)\> \,dx\nonumber\\
&=:I_1+I_2+I_3+I_4.
\end{align*}
We then derive the following estimates for $I_0$, $I_1$, $I_2$, $I_3$, and $I_4$
\begin{align*}
\int_{B_{1}}\eta^2
|\tau_h[V_p(\ov{\M}Dw)]|^2\,dx
\le cI_0,
\end{align*}
\begin{align*}
|I_1|
\le
\eps I_0
+
\frac{c|h|^{2\alpha}[a]_{C^{0,\alpha}(B_1)}^2}{\eps}
\int_{B_{1}}|\ov{\M}Dw|^{2q-p}\,dx,
\end{align*}
\begin{align*}
|I_2|
\le
|h|^{2\alpha}
\(\int_{B_1}|\ov{\M}Dv|^p\,dx
+
[a]_{C^{0,\alpha}(B_1)}^2
\int_{B_{1}}|\ov{\M}Dw|^{2q-p}\,dx\),
\end{align*}
\begin{align*}
|I_3|
\le
\eps I_0
+
\frac{c|h|^{2}}{\eps}
\int_{B_{1}}|\ov{\M}Dw|^{p}\,dx,
\end{align*}
and
\begin{align*}
|I_4|
\le
\eps I_0
+
\frac{c|h|^{2\alpha}}{\eps}
\int_{B_{1}}|\ov{\M}Dw|^{p}\,dx
+
\frac{c|h|^{2\alpha}\(\norm{a}_{L^\infty(B_1)}^2+[a]_{C^{0,\alpha}(B_1)}^2\)}{\eps}
\int_{B_{1}}|\ov{\M}Dw|^{2q-p}\,dx
\end{align*}
for any $\eps>0$ where $c=c(n,p,q,\nu,L)>0$, see Step 3 of the proof
of \cite[Theorem 5.1]{CM1} for further details. Combining
the estimates and choosing $\eps$ small enough, we conclude
\begin{align*}
&|h|^{-2\alpha}\int_{B_{2/3}}
|\tau_h[V_p(\ov{\M}Dw)]|^2\,dx\nonumber\\
&\le
c\int_{B_{1}}|\ov{\M}Dw|^{p}\,dx
+
c\(\norm{a}_{L^\infty(B_1)}^2+[a]_{C^{0,\alpha}(B_1)}^2\)
\int_{B_{1}}|\ov{\M}Dw|^{2q-p}\,dx,
\end{align*}
where $c=c(n,p,q,\nu,L)>0$. From the embedding of the fractional
Sobolev spaces in \cite[Section 5]{CM1}, we obtain that for every $\beta\in(0,\alpha)$
\begin{align*}
\norm{V_p(\ov{\M}Dw)}_{L^{\frac{2n}{n-2\beta}}(B_{1/2})}^2
&\le
c\norm{\ov{\M}Dw}_{L^p(B_{1})}^p
+
c\norm{V_p(\ov{\M}Dw)}_{L^{2}(B_{1})}^2\nonumber\\
&\quad\quad+
c\(\norm{a}_{L^\infty(B_1)}^2+[a]_{C^{0,\alpha}(B_1)}^2\)
\norm{\ov{\M}Dw}_{L^{2q-p}(B_1)}^{2q-p},
\end{align*}
where $c=c(n,p,q,\nu,L,\alpha,\beta)>0$. 
Now, by using covering argument and after scaling back this, as done in Step 4 of the proof
of \cite[Theorem 5.1]{CM1}, we have that
\begin{align}\label{454-b}
\norm{\ov{\M}Dw}_{L^{\frac{np}{n-2\beta}}(B_{t})}^p
&\le
\frac{c}{(s-t)^{2\beta}}
\norm{\ov{\M}Dw}_{L^p(B_{s})}^p\nonumber\\
&\quad\quad+
\frac{c}{(s-t)^{2\beta}}
\(\norm{a}_{L^\infty(B_{2r})}^2+r^{2\alpha}[a]_{C^{0,\alpha}(B_{2r})}^2\)
\norm{\ov{\M}Dw}_{L^{2q-p}(B_s)}^{2q-p}
\end{align}
for $0<r\le t<s\le2r$.

Next, for any fixed
\begin{align*}
 \sigma<\min\left\{\frac{\sigma_*}{2},\frac{2q}{p}-2,\frac{2\alpha}{n}\right\}, 
\end{align*}
where $\sigma_*$ in \eqref{809-a} and \eqref{810-a}, we let
$\beta\in(\frac{\alpha}{1+\sigma},\alpha)$ so that
$2q-p<\frac{np}{n-2\beta}$. Then thanks to the interpolation
inequality
\begin{align*}
\norm{\ov{\M}Dw}_{L^{2q-p}(B_s)} \le
\norm{\ov{\M}Dw}_{L^{p(1+\sigma)}(B_s)}^{1-\theta}
\norm{\ov{\M}Dw}_{L^{\frac{np}{n-2\beta}}(B_s)}^{\theta}
\end{align*}
for
$\theta=\frac{n(2q-2p-p\sigma)}{[2\beta-(n-2\beta)\sigma](2q-p)}$,
\eqref{454-b} leads to
\begin{align*}
\norm{\ov{\M}Dw}_{L^{\frac{np}{n-2\beta}}(B_{t})}^p
&\le
\frac{1}{2}\norm{\ov{\M}Dw}_{L^{\frac{np}{n-2\beta}}(B_{s})}^p
+
\frac{c}{(s-t)^{2\beta}}
\norm{\ov{\M}Dw}_{L^p(B_{2r})}^p\nonumber\\
&\hspace{0.23cm}  +
c\[\frac{1}{(s-t)^{2\beta}}
\(\norm{a}_{L^\infty(B_{2r})}^2+r^{2\alpha}[a]_{C^{0,\alpha}(B_{2r})}^2\)
\norm{\ov{\M}Dw}_{L^{p(1+\sigma)}(B_{2r})}^{(1-\theta)(2q-p)}\]^{\frac{p}{p-(2q-p)\theta}}
\end{align*}
for some $c=c(\data_0,C_0,\beta)>0$.
We here used Young's inequality with conjugate exponents
$\frac{p}{\theta(2q-p)}$ and $\frac{p}{p-\theta(2q-p)}$
as in the proof of \cite[Theorem 3]{MR3985927}.
Applying Lemma~\ref{lem:st} the previous estimate implies
\begin{align}\label{820}
\norm{\ov{\M}Dw}_{L^{\frac{np}{n-2\beta}}(B_{r})}^p
&\le
\frac{c}{r^{2\beta}}
\norm{\ov{\M}Dw}_{L^p(B_{2r})}^p\nonumber\\
&\hspace{0.23cm}  +
c\[\frac{1}{r^{2\beta}}
\(\norm{a}_{L^\infty(B_{2r})}^2+r^{2\alpha}[a]_{C^{0,\alpha}(B_{2r})}^2\)
\norm{\ov{\M}Dw}_{L^{p(1+\sigma)}(B_{2r})}^{(1-\theta)(2q-p)}\]^{\frac{p}{p-(2q-p)\theta}}
\end{align}
for some $c=c(\data_0,C_0,\beta)>0$. 
Moreover, by \eqref{808}, the
H\"older continuity of $a$ and $K\ge1$, we find that
\begin{align*}
\sup_{x\in B_{mr}}a(x)
\le(K+3m)r^\alpha [a]_{C^{0,\alpha}}
\le cK^\alpha r^\alpha [a]_{C^{0,\alpha}}
\quad
\mbox{for $m\in(1,4)$},
\end{align*}
from which \eqref{820} yields
\begin{align*}
&\(\bint_{B_r}|\ov{\M}Dw|^{\frac{np}{n-2\beta}}\,dx\)^{\frac{n-2\beta}{np}}\nonumber\\
&\le
c\[1+(K^2[a]_{C^{0,\alpha}}^2)^{b_1/p}
\norm{\ov{\M}Dw}_{L^{p(1+\sigma)}(B_{2r})}^{\frac{b_2(1+\sigma)-1}{p(1+\sigma)}}\]
\(\bint_{B_{2r}}|\ov{\M}Dw|^{p(1+\sigma)}\,dx\)^{\frac{1}{p(1+\sigma)}}\nonumber\\
&\le
c\(\bint_{B_{2r}}|\ov{\M}Dw|^{p(1+\sigma)}\,dx\)^{\frac{1}{p(1+\sigma)}},
\end{align*}
where $b_1=\frac{p}{p-(2q-p)\theta}$,
$b_2=\frac{(1-\theta)(2q-p)}{(1+\sigma)[p-(2q-p)\theta]}$ and
$c=c(\data_0,C_0,\beta,K)$, see the proof of \cite[Theorem
5]{MR3985927} for details. In view of \eqref{809-a},
we finally obtain
\begin{align*}
\(\bint_{B_r}|\ov{\M}Dw|^{\frac{np}{n-2\beta}}\,dx\)^{\frac{n-2\beta}{np}}
\le
c\(\bint_{B_{4r}}H(x,\ov{\M}Dw)\,dx\)^{\frac{1}{p}}
\end{align*}
for some $c=c(\data,C_0,\beta,K)>0$. Making a suitable choice of
$\beta$ and using H\"older's inequality, we obtain \eqref{809}.
Moreover, the estimate \eqref{810} follows by \eqref{809} and a
covering argument.
\end{proof}

\subsection{Reference problem 3}

We finally consider the weak solution $v$ to the problem
\begin{align}\label{1001}
\begin{cases}
\dive \big(\ov{\M}\ov{ A}_p(\ov{\M} Dv)+a_o\ov{\M} \ov{A}_q(\ov{\M} Dv)\big)
=0
\quad\mbox{in $B$},\\
v=v_0
\quad\mbox{on $\partial B$,}
\end{cases}
\end{align}
where $a_o\ge0$ is a fixed constant.
Like the proof of Lemma~\ref{lem:h-energy}, we have the energy estimate
\begin{align*}
\bint_{B}(|\ov{\M}Dv|^p+a_o|\ov{\M}Dv|^q)\,dx
\le
c\bint_{B}(|\ov{\M}Dv_0|^p+a_o|\ov{\M}Dv_0|^q)\,dx
\end{align*}
for some $c=c(n,p,q,\nu,L)>0$.
We now provide regularity results for the weak solution $v$.

\begin{lemma}\label{lem:v-lip}
Let $v$ be the weak solution to \eqref{1001}.
Then we have
\begin{align*}
\sup_{\frac{1}{2}B}\(|\ov{\M}Dv|^p+a_o|\ov{\M}Dv|^q\)
\le
c\bint_{B}(|\ov{\M}Dv|^p+a_o|\ov{\M}Dv|^q)\,dx,
\end{align*}
where  $c=c(n,p,q,\nu,L,\Lambda)>0.$
\end{lemma}
\begin{proof}
Defining
\begin{align}\label{1004}
\Phi_o(\eta):=\frac{\ov{\M}\ov{A}_p(\ov{\M}\eta)}{|\ov{\M}|^p} +
\tilde{a}_o\frac{\ov{\M}\ov{A}_q(\ov{\M}\eta)}{|\ov{\M}|^q}
\end{align}
with $\tilde{a}_o=a_o|\overline{\M}|^{q-p}$, we have from
\eqref{cond-M} and \eqref{ApAq} that
\begin{align} \label{1005}
\begin{cases}
|\Phi_o(\eta)|+|D\Phi_o(\eta)||\eta|
\le
L(|\eta|^{p-1}+\tilde{a}_0|\eta|^{q-1}),\\
\<D\Phi_o(\eta)\zeta,\zeta\>\ge\tilde{\nu}(|\eta|^{p-2}+\tilde{a}_o|\eta|^{q-2})|\zeta|^2
\end{cases}
\end{align}
for $\tilde{\nu}=\frac{\nu}{\Lambda^{q+1}}$.
We observe that $v$ solves the problem
\begin{align*}
\dive\, \Phi_o(Dv)=0 \quad\mbox{in $B$}
\end{align*}
and from \cite[Theorem 2.1]{CM3}
\begin{align}\label{1007}
\sup_{\frac{1}{2}B}(|Dv|^p+\tilde{a}_o|Dv|^q)
\le
c\bint_{B}(|Dv|^p+\tilde{a}_o|Dv|^q)\,dx
\end{align}
for some constant $c$ depending only on $n,p,q,\nu,L,\Lambda$, in
particular, being independent of $\tilde{a}_o$. Multiplying
$|\ov{\M}|^{p}$ on both sides in \eqref{1007} and applying
\eqref{cond-M}, we have the conclusion.
\end{proof}

\begin{lemma}\label{lem:v-cz}
Let $d>1$ and $v$ be the weak solution to \eqref{1001}.
If $|\ov{\M}Dv_0|^p+a_o|\ov{\M}Dv_0|^q\in L^d(B)$, then
we have $|\ov{\M}Dv|^p+a_o|\ov{\M}Dv|^q\in L^d(B)$
with the estimate
\begin{align*}
\bint_{B}\(|\ov{\M}Dv|^p+a_o|\ov{\M}Dv|^q\)^d\,dx
\le
c\bint_{B}\(|\ov{\M}Dv_0|^p+a_o|\ov{\M}Dv_0|^q\)^d\,dx
\end{align*}
where $c=c(n,p,q,\nu,L,\Lambda,d)>0$.
\end{lemma}
\begin{proof}
As in the proof of Lemma~\ref{lem:v-lip},
$v$ is the weak solution to the problem
\begin{align*}
\begin{cases}
\dive\, \Phi_o(Dv)=0 \quad\mbox{in $B$},\\
v=v_0\quad\mbox{on $\partial B$},
\end{cases}
\end{align*}
where $\Phi_o$ is given in \eqref{1004} with \eqref{1005}.
We have from \cite[Theorem 2]{MR3985927}
\begin{align*}
\bint_{B}\(|Dv|^p+\tilde{a}_o|Dv|^q\)^d\,dx
\le
c\bint_{B}\(|Dv_0|^p+\tilde{a}_o|Dv_0|^q\)^d\,dx,
\end{align*}
where $\tilde{a}_o=a_o|\overline{\M}|^{q-p}$ and $c=c(n,p,q,\nu,L,\Lambda,d)>0$.
We here point out that the constant $c$ is independent of $\tilde{a}_o$ and $B=B_\rho$ from $\rho \in(0,1]$.
Multiplying $|\ov{\M}|^{pd}$ on both sides in the previous inequality and using \eqref{cond-M},
we finish the proof.
\end{proof}

\section{Proof of Theorem~\ref{thm}}\label{sec5}

We are now ready to start the proof of Theorem~\ref{thm} following
the same spirit in the earlier papers \cite{AM07, MR3985927, CM3,
MR4361836}.

\begin{proof}[Proof of Theorem~\ref{thm}]
We assume that $(\log\M,A_p,A_q)$ is $\delta$-vanishing, namely,
\begin{align}\label{delta-ass}
\abs{\log\M}_{\BMO(\Omega)}\le\delta
\quad
\mbox{and}
\quad
\sup_{B\subset\Omega}
\bint_{B}\(\Theta_p+\Theta_q\)[B](x)\,dx\le\delta,
\end{align}
where $B$ is any ball contained in $\Omega$. 
Here $\delta>0$ will be determined later as a
sufficiently small value less than $W$ given in \eqref{601},
depending only on $\data$, $\gamma$, $\hat{\tau}$ and $\|H(\cdot,\M F)\|_{L^\gamma(\Omega_{\hat{\tau}})}$. Our proof proceeds in the
following steps.

{\it Step1: Exit time and covering.}
Let $B_R\subset\Omega_{2\hat{\tau}}$ be a ball with radius $R\le r_o$,
where $r_o\le1$ will be determined later, depending only on $\data$, $\gamma$, $\hat{\tau}$ and $\|H(\cdot,\M F)\|_{L^\gamma(\Omega_{\hat{\tau}})}$.
We first fix two arbitrary radii $\rho_1,\rho_2$ such that $R/2\le\rho_1<\rho_2\le R$,
and consider the upper level sets
\begin{equation*}
  E(r,\lambda):=\{x\in B_r: H(x,\M(x) Du(x))>\lambda\}  
\end{equation*}
for $\lambda>0$
and for balls $B_r$ which are concentric to $B_R$ with $r\in[R/2,R]$.
We observe that for a.e. $x_0\in E(\rho_1,\lambda)$
\begin{align}\label{1103}
\lim_{r\rightarrow0}
\bint_{B_r(x_0)}\[H(x,\M Du)
+
\frac{1}{\delta}[H(x,\M F)\]\,dx
>\lambda.
\end{align}
On the other hand, putting
\begin{align}\label{1105-2}
\lambda_R:=\bint_{B_R}\[H(x,\M Du)
+
\frac{1}{\delta}[H(x,\M F)\]\,dx,
\end{align}
we have for a.e. $x_0\in E(\rho_1,\lambda)$ and for $r\in(\frac{\rho_2-\rho_1}{40},\rho_2-\rho_1]$
\begin{align}\label{1104}
\bint_{B_r(x_0)}\[H(x,\M Du)
+
\frac{1}{\delta}[H(x,\M F)\]\,dx 
&<
\(\frac{40R}{\rho_2-\rho_1}\)^n
\bint_{B_{R}}\[H(x,\M Du)
+
\frac{1}{\delta}[H(x,\M F)\]\,dx\nonumber\\
&= \(\frac{40R}{\rho_2-\rho_1}\)^n \lambda_R =: \lambda_0.
\end{align}
Thanks to \eqref{1103}, \eqref{1104} and the absolute continuity of
the integral, we can select the maximal radius
$r_{x_0}\in(0,\frac{\rho_2-\rho_1}{40}]$ for a.e $x_0\in
E(\rho_1,\lambda)$ with $\lambda\ge\lambda_0$ so that
\begin{align*}
\bint_{B_{r_{x_0}}(x_0)}\[H(x,\M Du)
+
\frac{1}{\delta}[H(x,\M F)\]\,dx=\lambda
\end{align*}
and
\begin{align*}
 \bint_{B_{r}(x_0)}\[H(x,\M Du)
+
\frac{1}{\delta}[H(x,\M F)\]\,dx<\lambda
\quad\mbox{for all $r\in(r_{x_0},\rho_2-\rho_1]$.}
\end{align*}

Hereafter we always assume that $\lambda\ge\lambda_0$.
Using a Vitali type covering lemma, we therefore have
a countable family of mutually disjoint balls $B_{r_{x_i}}(x_i)$ with $x_i\in E(\rho_1,\lambda)\setminus \text{\it negligible set}$ and $r_{x_i}\in(0,\frac{\rho_2-\rho_1}{40}]$
such that
\begin{align}\label{1107-1}
\bint_{B_{r_{x_i}}(x_i)}\[H(x,\M Du)
+
\frac{1}{\delta}[H(x,\M F)\]\,dx
=
\lambda,
\end{align}
\begin{align*}
\bint_{B_r(x_i)}\[H(x,\M Du)
+
\frac{1}{\delta}[H(x,\M F)\]\,dx<\lambda
\quad\mbox{for all $r\in(r_{x_i},\rho_2-\rho_1]$},
\end{align*}
and
\begin{align}\label{1109-1}
E(\rho_1,\lambda)\setminus negligible\ set\subset\bigcup_{i\in\N}B_{5r_{x_i}}(x_i)\subset B_{\rho_2}.
\end{align}
In what follows, we abbreviate $r_i=r_{x_i}$ and
$B_i^{(l)}:=B_{lr_{x_i}}(x_i)$ for simplicity. In particular, it follows that
\begin{align*}
40r_{i}\le\rho_2-\rho_1\le R\le r_o\le1,
\quad
B_i^{(40)}\subset B_{\rho_2}
\end{align*}
and
\begin{align}\label{lambda-bdd}
\bint_{B_i^{(40)}}\[H(x,\M Du)
+
\frac{1}{\delta}[H(x,\M F)\]\,dx<\lambda
\end{align}
for all $i\in\N$.

\

{\it Step 2: Comparison estimates.}
Recalling $H(x,\M Du)\in L^1(\Omega)$,
we first consider the weak solution $h_i$ to the problem
\begin{align}\label{h-eq}
\begin{cases}
\dive\( \M A_p(x,\M Dh_i)+a(x)\M A_q(x,\M Dh_i) \)=0
\quad\mbox{in $B_i^{(40)}$},\\
h_i= u \quad\mbox{on $\partial B_i^{(40)}$}.
\end{cases}
\end{align}
By Lemma~\ref{lem:h-energy}, we have the energy estimate
\begin{align}\label{1111}
\int_{B_i^{(40)}}H(x,\M Dh_i)\,dx
\le
c\int_{B_i^{(40)}}H(x,\M Du)\,dx
\end{align}
for some $c=c(n,p,q,\nu,L)>0$.

We now make a comparison estimate between $u$ and $h_i$ to find that
\begin{align}\label{u-h}
&\bint_{B_i^{(40)}}\(|V_p(\M Du)-V_p(\M Dh_i)|^2+a(x)|V_q(\M Du)-V_q(\M Dh_i)|^2\)\dx\nonumber\\
&\le
\eps\bint_{B_i^{(40)}}H(x,\M Du)\,dx
+
c_\eps\bint_{B_i^{(40)}}H(x,\M F)\,dx
\end{align}
for any $\eps\in(0,1)$, where $c_\eps=c_\eps(n,p,q,\nu,L,\eps)>0$.
In view of Lemma~\ref{lem:testing}, we use $\varphi=u-h_i$ as a test
function to \eqref{eq-u} and \eqref{h-eq}, respectively, to get
\begin{align*}
&\bint_{B_i^{(40)}} \<A_p(x,\M Du)-A_p(x,\M Dh_i),\ \M Du-\M Dh_i\>\,dx\nonumber\\
&\quad+
\bint_{B_i^{(40)}} a(x)\<A_q(x,\M Du)-A_q(x,\M Dh_i),\ \M Du-\M Dh_i\>\,dx\nonumber\\
&= \bint_{B_i^{(40)}}\<G(x,\M F),\ \M Du-\M Dh_i\>\, dx.
\end{align*}
Using \eqref{v-v}, \eqref{G}, Young's inequality and \eqref{1111},
we then discover that for any $\eps>0$
\begin{align*}
&\bint_{B_i^{(40)}}\(|V_p(\M Du)-V_p(\M Dh_i)|^2+a(x)|V_q(\M Du)-V_q(\M Dh_i)|^2\)\dx\nonumber\\
&\le
c_\eps\bint_{B_i^{(40)}}H(x,\M F)\,dx
+
\tfrac{\eps}{2}\bint_{B_i^{(40)}}H(x,\M Du)\,dx
+
\tfrac{\eps}{2}\bint_{B_i^{(40)}}H(x,\M Dh_i)\,dx\nonumber\\
&\le
c_\eps\bint_{B_i^{(40)}}H(x,\M F)\,dx
+
\eps\bint_{B_i^{(40)}}H(x,\M Du)\,dx,
\end{align*}
where $c_\eps=c_\eps(n,p,q,\nu,L,\eps)>0$,
which yields \eqref{u-h}.

With $\ov{\M}=\(\M\)_{B_i^{(20)}}^{\log}$ and $\tilde{\sigma}>0$ as in Proposition~\ref{prop:u-higher} we see from
Lemma~\ref{cor:h}, Lemma~\ref{lem:h-higher} and \eqref{1111} that
there is a constant $\sigma_1=\sigma_1(\data)\in(0,\tilde{\sigma})$ such that
for any $\sigma\in(0,\sigma_1]$
\begin{align}\label{1115}
\bint_{B_i^{(20)}}\[H(x,\ov{\M}Dh_i)\]^{1+\frac{\sigma}{2}}\,dx
&\le
c\(\bint_{B_i^{(20)}}\[H(x,\M Dh_i)\]^{1+\sigma}\,dx\)^{\(1+\frac{\sigma}{2}\)\frac{1}{1+\sigma}}\nonumber\\
&\le
c\(\bint_{B_i^{(40)}}H(x,\M Dh_i)\,dx\)^{1+\frac{\sigma}{2}}\nonumber\\
& \le 
c\(\bint_{B_i^{(40)}}H(x,\M Du)\,dx\)^{1+\frac{\sigma}{2}}
\end{align}
and
\begin{align}\label{norm-h}
\|H(\cdot,\ov{\M}Dh_i)\|_{L^{1+\frac{\sigma}{2}}({B_i^{(20)}})}
&\le
c|B_i^{(20)}|^{\frac{\sigma}{(2+\sigma)(1+\sigma)}}
\|H(\cdot,\M Dh_i)\|_{L^{1+\sigma}({B_i^{(20)}})}\nonumber\\
&\le
cr_i^{\frac{n\sigma}{(2+\sigma)(1+\sigma)}}
\|H(\cdot,\M Dh_i)\|_{L^{1+\sigma}({B_i^{(40)}})}\nonumber\\
&\le
c\|H(\cdot,\M Du)\|_{L^{1+\sigma}({B_i^{(40)}})},
\end{align}
where $c=c(\data,\sigma)>0$, provided
that $\delta>0$ is selected sufficiently small, depending only on
$\data$.
In the last inequality we used $r_i<1$.

\

We next consider the weak solution $k_i$ to the problem
\begin{align}\label{k-eq}
\begin{cases}
\dive \( \ov{\M} A_p(x,\ov{\M} Dk_i)+a(x)\ov{\M} A_q(x,\ov{\M} Dk_i)\)
=
0
\quad\mbox{in $B_i^{(20)}$,}\\
k_i=h_i\quad\mbox{on $\partial B_i^{(20)}$},
\end{cases}
\end{align}
where $h_i$ is the weak solution to \eqref{h-eq}.
We have from Lemma \ref{lem:h-energy} and \eqref{1115} that
\begin{align}\label{1117}
\bint_{B_i^{(20)}}H(x,\ov{\M} Dk_i)\,dx
\le
c\bint_{B_i^{(20)}}H(x,\ov{\M} Dh_i)\,dx
& \le 
c\bint_{B_i^{(40)}}H(x,\M Du)\,dx
\end{align}
for some $c=c(\data)>0$.
Furthermore, in view of Remark~\ref{rem:h}, Lemma~\ref{lem:h-higher}, \eqref{1115} and \eqref{norm-h} 
there is a constant $\sigma_{2}\in(0,\frac{\sigma_1}{2})$ such that for any $\sigma\in(0,\sigma_2]$
\begin{align}\label{1118}
\bint_{B_i^{(20)}}\[H(x,\ov{\M}Dk_i)\]^{1+\sigma}\,dx
&\le
c\bint_{B_i^{(20)}}\[H(x,\ov{\M}Dh_i)\]^{1+\sigma}\,dx\nonumber\\
&\le 
c\(\bint_{B_i^{(40)}}H(x,\M Du)\,dx\)^{1+\sigma}
\end{align}
and
\begin{align}\label{norm-k}
\|H(\cdot,\ov{\M}Dk_i)\|_{L^{1+\sigma}(B_i^{(20)})}
&\le
c\|H(\cdot,\ov{\M}Dh_i)\|_{L^{1+\sigma}(B_i^{(20)})} \nonumber\\
&\le
c\|H(\cdot,\M Du)\|_{L^{1+2\sigma}(B_i^{(40)})},
\end{align}
where positive
numbers $\sigma_{2},c$ depend only on $\data$, if $\delta>0$ is small enough. Moreover,
\eqref{1118} yields $H(x,\M Dk_i)\in L^{1}(B_i^{(20)})$ from
Remark~\ref{rem:h2}, if $\delta$ is much smaller. This allows us to
use $\varphi=h_i-k_i$ as a test function in the weak formulation of
\eqref{h-eq} and \eqref{k-eq} to find
\begin{align*}
&\bint_{B_i^{(20)}} \<A_p(x,\M Dh_i)-A_p(x,\ov{\M}Dk_i),\ \M Dh_i-\ov{\M}Dk_i\>\,dx\nonumber\\
&\quad\quad+
\bint_{B_i^{(20)}} a(x)\<A_q(x,\M Dh_i)-A_q(x,\ov{\M}Dk_i), \ \M Dh_i-\ov{\M}Dk_i\>\,dx\nonumber\\
&=
\bint_{B_i^{(20)}} \<A_p(x,\M Dh_i)+a(x) A_q(x,\M Dh_i),\ (\M-\ov{\M}) Dk_i\>\,dx\nonumber\\
&\quad\quad-
\bint_{B_i^{(20)}}\<A_p(x,\ov{\M}Dk_i)+a(x)A_q(x,\ov{\M}Dk_i),\ (\M-\ov{\M})Dh_i\>\,dx.
\end{align*}
We then apply \eqref{v-v}, \eqref{ApAq} and Young's inequality,
which follows
that for any $\eps>0$
\begin{align}\label{1120}
&\bint_{B_i^{(20)}}\(|V_p(\M Dh_i)-V_p(\ov{\M} Dk_i)|^2+a(x)|V_q(\M Dh_i)-V_q(\ov{\M} Dk_i)|^2\)\dx\nonumber\\
&\le c
\bint_{B_i^{(20)}} \(|\M Dh_i|^{p-1}+a(x) |\M Dh_i|^{q-1}\) |(\M-\ov{\M}) Dk_i|\,dx\nonumber\\
&\quad\quad+
\bint_{B_i^{(20)}}\(|\ov{\M}Dk_i|^{p-1}+a(x)|\ov{\M}Dk_i|^{q-1}\) |(\M-\ov{\M})Dh_i|\,dx.\nonumber\\
&\le
\eps\bint_{B_i^{(20)}}H(x,\M Dh_i)\,dx
+
c_\eps\bint_{B_i^{(20)}}
\[\(\frac{|\M-\ov{\M}|}{|\ov{\M}|}\)^p
 +
\(\frac{|\M-\ov{\M}|}{|\ov{\M}|}\)^q\]
H(x,\ov{\M} Dk_i)\,dx\nonumber\\
&\quad\quad+
\eps\bint_{B_i^{(20)}}H(x,\ov{\M} Dk_i)\,dx
+
c_\eps\bint_{B_i^{(20)}}
\[\(\frac{|\M-\ov{\M}|}{|\ov{\M}|}\)^p
 +
\(\frac{|\M-\ov{\M}|}{|\ov{\M}|}\)^q\]
H(x,\ov{\M} Dh_i)\,dx,
\end{align}
where $c_\eps(n,p,q,\nu,L,\eps)>0$. Meanwhile, using Young's
inequality, Lemma~\ref{lem:m-m} and $\eqref{delta-ass}$, we have
\begin{align}\label{1121}
&\bint_{B_i^{(20)}}
\[\(\frac{|\M-\ov{\M}|}{|\ov{\M}|}\)^p
 +
\(\frac{|\M-\ov{\M}|}{|\ov{\M}|}\)^q\]
H(x,\ov{\M} V_i)\,dx\nonumber\\
&\le
c\(\bint_{B_i^{(20)}}
\[\(\frac{|\M-\ov{\M}|}{|\ov{\M}|}\)^p
 +
\(\frac{|\M-\ov{\M}|}{|\ov{\M}|}\)^q\]^{s'}\,dx\)^{\frac{1}{s'}}
\(\bint_{B_i^{(20)}}
\[H(x,\ov{\M} V_i)\]^{s}\,dx\)^{\frac{1}{s}}\nonumber\\
&\le
c\delta^p\(\bint_{B_i^{(20)}}
\[H(x,\ov{\M} V_i)\]^{s}\,dx\)^{\frac{1}{s}}\nonumber\\
&\le
c\delta^p\bint_{B_i^{(40)}}
H(x,\M Du)\,dx,
\end{align}
provided that $\delta>0$ is smaller, where $V_i\in\{Dk_i,Dh_i\}$,
$s=1+\sigma_{2}$ for $\sigma_{2}>0$ as in \eqref{1118} and
$c=c(\data)>0$. For the last inequality we used \eqref{1115}
and \eqref{1118}. We then combine \eqref{1111}, \eqref{1117},
\eqref{1120} and \eqref{1121} to deduce
\begin{align}\label{h-k}
&\bint_{B_i^{(20)}}\(|V_p(\M Dh_i)-V_p(\ov{\M} Dk_i)|^2+a(x)|V_q(\M Dh_i)-V_q(\ov{\M} Dk_i)|^2\)\,dx\nonumber\\
&\le
(c\eps+c_\eps\delta^p)\bint_{B_i^{(40)}}
H(x,\M Du)\,dx
\end{align}
for any $\eps>0$,
where both $c$ and $c_\eps$ depend on  $\data$,
while $c_\eps$ further depends on $\eps$ as well.

\

With
\begin{align*}
\ov{A}_p(\xi)=\bint_{B_i^{(20)}}A_p(x,\xi)\,dx\quad\mbox{and}\quad\ov{A}_q(\xi)=\bint_{B_i^{(20)}}A_q(x,\xi)
\ dx,    
\end{align*}
we next consider the weak solution $w_i$ to the problem
\begin{align}\label{w-eq}
\begin{cases}
\dive \(\ov{\M} \ov{A}_p(\ov{\M} Dw_i)+a(x)\ov{\M} \ov{A}_q(\ov{\M} Dw_i)\)
=
0
\quad\mbox{in $B_i^{(20)}$,}\\
w_i=k_i\quad\mbox{on $\partial B_i^{(20)}$},
\end{cases}
\end{align}
where $k_i$ is the weak solution to \eqref{k-eq}.
From \eqref{436-d} and \eqref{1117} we have 
\begin{align}\label{1124}
\bint_{B_i^{(20)}}H(x,\ov{\M}Dw_i)\,dx
\le
c\bint_{B_i^{(20)}}H(x,\ov{\M}Dk_i)\,dx
\le
c\bint_{B_i^{(40)}}H(x,\M Du)\,dx
\end{align}
for some $c=c(\data)>0$.
In addition, we see from \eqref{810-a}, \eqref{1118} and \eqref{norm-k}
\begin{align}\label{1126-1}
\bint_{B_i^{(20)}}[H(x,\ov{\M}Dw_i)]^{1+\sigma_{3}}\,dx
&\le
c\bint_{B_i^{(20)}}[H(x,\ov{\M}Dk_i)]^{1+\sigma_{3}}\,dx\nonumber\\
& \le 
c\(\bint_{B_i^{(40)}}H(x,\M Du)\,dx\)^{1+\sigma_{3}}
\end{align}
and
\begin{align}\label{norm-w}
\|H(\cdot,\ov{\M}Dw_i)\|_{L^{1+\sigma_{3}}(B_i^{(20)})}
&\le
c\|H(\cdot,\ov{\M}Dk_i)\|_{L^{1+\sigma_{3}}(B_i^{(20)})}\nonumber\\
&\le
c\|H(\cdot,\M Du)\|_{L^{1+2\sigma_{3}}(B_i^{(40)})}
\end{align}
for some $\sigma_3=\sigma_3(\data)\in(0,\sigma_2)$ and $c=c(\data)>0$.

Now we make the comparison estimate between $w_i$ and $k_i$. Testing
\eqref{w-eq} and \eqref{k-eq} with $\varphi = k_i-w_i$, we deduce
\begin{align*}
&\bint_{B_i^{(20)}}
\<
\ov{A}_p(\ov{\M}Dk_i)-\ov{A}_p(\ov{\M}Dw_i),\ov{\M}Dk_i-\ov{\M}Dw_i
\>\,dx\nonumber\\
&\quad\quad+
\bint_{B_i^{(20)}}
a(x)\<\ov{A}_q(\ov{\M}Dk_i)-\ov{A}_q(\ov{\M}Dw_i),
\ov{\M}Dk_i-\ov{\M}Dw_i\>\,dx\nonumber\\
&=
\bint_{B_i^{(20)}}
\<
\ov{A}_p(\ov{\M}Dk_i)-A_p(x,\ov{\M}Dk_i),\ov{\M}Dk_i-\ov{\M}Dw_i
\>\,dx\nonumber\\
&\quad\quad+
\bint_{B_i^{(20)}}
a(x)\<\ov{A}_q(\ov{\M}Dk_i)-A_q(x,\ov{\M}Dk_i),
\ov{\M}Dk_i-\ov{\M}Dw_i\>\,dx.
\end{align*}
We then apply \eqref{v-v}, \eqref{Theta} and Young's inequality to get
\begin{align}\label{1126}
\bint_{B_i^{(20)}}
&\(|V_p(\ov{\M}Dk_i)-V_p(\ov{\M}Dw_i)|^2
+
a(x)|V_q(\ov{\M}Dk_i)-V_q(\ov{\M}Dw_i)|^2\)\,dx\nonumber\\
&\le
c\bint_{B_i^{(20)}}
\Theta_p[B_i^{(20)}](x)|\ov{\M}Dk_i|^{p-1}|\ov{\M}Dk_i-\ov{\M}Dw_i| \, dx\nonumber\\
&\quad\quad+
c\bint_{B_i^{(20)}}
a(x)\Theta_q[B_i^{(20)}](x)||\M Dk_i|^{q-1}|\ov{\M}Dk_i-\ov{\M}Dw_i|\,dx\nonumber\\
&\le c_\eps\bint_{B_i^{(20)}}
(\Theta_p+\Theta_q)[B_i^{(20)}](x)H(x,\ov{\M}Dk_i)\,dx +
\eps\bint_{B_i^{(20)}} H(x,\ov{\M}Dw_i)\ dx
\end{align}
for any $\eps>0$, where $c_\eps=c_\eps(n,p,q,\nu,L,\Lambda,\eps)>0$.
For the last inequality we used $\Theta_p,\Theta_q\le 2L$.
Meanwhile, using the fact that $\Theta_p,\Theta_q\le 2L$, \eqref{1118} and \eqref{delta-ass}, we find that
\begin{align*}
&\bint_{B_i^{(20)}}
(\Theta_p+\Theta_q)[B_i^{(20)}](x)H(x,\ov{\M}Dk_i)\,dx\nonumber\\
&\le
c\[\bint_{B_i^{(20)}}
(\Theta_p+\Theta_q)[B_i^{(20)}](x)\,dx\]^{\frac{\sigma_{2}}{1+\sigma_{2}}}
\[\bint_{B_i^{(20)}}
H(x,\ov{\M}Dk_i)^{1+\sigma_{2}}\,dx\]^{\frac{1}{1+\sigma_{2}}}\nonumber\\
&\le c\delta^{\frac{\sigma_{2}}{1+\sigma_{2}}} \bint_{B_i^{(40)}}
H(x,\M Du)\ dx
\end{align*}
for some $c=c(\data)>0$.
Inserting this into \eqref{1126} and using \eqref{1124}, we have that
\begin{align}\label{k-w}
\bint_{B_i^{(20)}}
&\(|V_p(\ov{\M}Dk_i)-V_p(\ov{\M}Dw_i)|^2
+
a(x)|V_q(\ov{\M}Dk_i)-V_q(\ov{\M}Dw_i)|^2\)\,dx\nonumber\\
&\le (c\eps+c_\eps\delta^{\frac{\sigma_{2}}{1+\sigma_{2}}})
\bint_{B_i^{(40)}} H(x,\M Du)\,dx
\end{align}
for any $\eps>0$, where $c_\eps=c_\eps(\data,\eps)>0$.

\

We now find point $x_{i,o}\in\ov{B_i^{(10)}}$ such that
\begin{align*}
a_{i,o}:=a(x_{i,o})=\sup_{x\in B_i^{(10)}}a(x).
\end{align*}
We finally consider the weak solution $v_i$ to the problem
\begin{align}\label{v-eq}
\begin{cases}
\dive \(\ov{\M}\ov{ A}_p(\ov{\M} Dv_i)+a_{i,o}\ov{\M} \ov{A}_q(\ov{\M} Dv_i)\)
=0
\quad\mbox{in $B_i^{(10)}$},\\
v_i=w_i\quad\mbox{on $\partial B_i^{(10)}$},
\end{cases}
\end{align}
where $w_i$ is the weak solution to \eqref{w-eq}.
We then have the energy estimate
\begin{align}\label{1132}
\bint_{B_i^{(10)}}H(x_{i,o},\ov{\M} Dv_i)\,dx
\le
c\bint_{B_i^{(10)}}H(x_{i,o},\ov{\M} Dw_i)\,dx
\end{align}
for some $c=c(n,p,q,\nu,L)>0$.
Here, we point out that $Dw_i\in L^q(B_i^{(10)})$ by Lemma~\ref{lem:w-reg}.

We derive our final comparison estimate between $v_i$ and $w_i$.
Our approach is based on the same idea in \cite{MR3985927},
but we consider gradients with the constant matrix weight $\ov{\M}$
and use the results in the previous section.
Now, using $\varphi=v_i-w_i$ in the weak formulations of \eqref{w-eq} and \eqref{v-eq},
we obtain
\begin{align*}
&\bint_{B_i^{(10)}}
\<
\ov{A}_p(\ov{\M}Dv_i)-\ov{A}_p(\ov{\M}Dw_i),\ov{\M}Dv_i-\ov{\M}Dw_i
\>\,dx\nonumber\\
&\quad\quad+
\bint_{B_i^{(10)}}
a_{i,o}\<\ov{A}_q(\ov{\M}Dv_i)-\ov{A}_q(\ov{\M}Dw_i),
\ov{\M}Dv_i-\ov{\M}Dw_i\>\,dx\nonumber\\
&=
\bint_{B_i^{(10)}}
(a(x)-a_{i,o})\<\ov{A}_q(\ov{\M}Dw_i),
\ov{\M}Dv_i-\ov{\M}Dw_i\>\,dx.
\end{align*}
This gives by \eqref{v-v}
\begin{align}\label{1134}
&\bint_{B_i^{(10)}}
\(|V_p(\ov{\M}Dw_i)-V_p(\ov{\M}Dv_i)|^2
+
a(x)|V_q(\ov{\M}Dw_i)-V_q(\ov{\M}Dv_i)|^2\)\,dx\nonumber\\
&\le
\bint_{B_i^{(10)}}
\(|V_p(\ov{\M}Dw_i)-V_p(\ov{\M}Dv_i)|^2
+
a_{i,o}|V_q(\ov{\M}Dw_i)-V_q(\ov{\M}Dv_i)|^2\)\,dx\nonumber\\
&\le
c\(\osc_{B_i^{(10)}}a\)
\bint_{B_i^{(10)}}|\ov{\M}Dw_i|^{q-1}
|\ov{\M}Dv_i-\ov{\M}Dw_i|\,dx =: {(\rm{I})}
\end{align}
for some $c=c(n,p,q,\nu,L)>0$.
To estimate ({\rm{I}}),
we proceed by considering two cases:
\begin{align}\label{1135}
\inf_{B_i^{(10)}}a(x)>K[a]_{0,\alpha}r_{x_i}^{\alpha}
\end{align}
and
\begin{align}\label{1136}
\inf_{B_i^{(10)}}a(x)\le K[a]_{0,\alpha}r_{x_i}^{\alpha},
\end{align}
where $K>20$ is determined later, depending only on $\data$, $\gamma$, $\hat{\tau}$ and $\|H(\cdot,\M F)\|_{L^{\gamma}(\Omega_{\hat{\tau}})}$.

We first consider the case of \eqref{1135}, referred to the $(p,q)$-phase.
We then obtain for all $x\in B_i^{(10)}$
\begin{align}\label{1137}
\osc_{B_i^{(10)}}a
\le
20 [a]_{0,\alpha}r_{i}^\alpha
\le
\frac{20a(x)}{K},
\end{align}
which induces that
\begin{align}\label{1138}
a(x)
\le
a_{i,o}
\le
a(x)+\osc_{B_i^{(10)}}a
\le
2a(x).
\end{align}
From \eqref{1124}, \eqref{1132}, \eqref{1137} and \eqref{1138}, we
deduce
\begin{align}\label{1140}
({\rm{I}})
\le
\frac{c}{K}\bint_{B_i^{(10)}}H(x_{i,o},\ov{\M} Dw_i)\,dx
\le
\frac{c_1}{K}\bint_{B_i^{(40)}}H(x,\M Du)\,dx
\end{align}
for some $c_1=c_1(\data)>0$. 
This yields by \eqref{1134}
\begin{align*}
&\bint_{B_i^{(10)}}
\(|V_p(\ov{\M}Dw_i)-V_p(\ov{\M}Dv_i)|^2
+
a(x)|V_q(\ov{\M}Dw_i)-V_q(\ov{\M}Dv_i)|^2\)\,dx\nonumber\\
&\le
\frac{c_1}{K}\bint_{B_i^{(40)}}H(x,\M Du)\,dx.
\end{align*}
In particular, the second inequality in
\eqref{1140} arrives at
\begin{align}\label{1142-1}
\bint_{B_i^{(10)}}H(x_{i,o},\ov{\M} Dw_i)\,dx
\le
c\bint_{B_i^{(40)}}H(x,\M Du)\,dx
\end{align}
for some $c=c(\data)>0$.

We next consider the case of \eqref{1136}, referred to the $p$-phase.
We observe
\begin{align}\label{1142}
a_{i,o}
\le
20 [a]_{0,\alpha}r_{i}^\alpha
+
\inf_{B_i^{(10)}}a(x)
\le
2K[a]_{0,\alpha}r_{i}^\alpha,
\end{align}
from which we apply Lemma~\ref{lem:w-higher}, to have
\begin{align}\label{1143}
\(\bint_{B_i^{(10)}}
|\ov{\M} Dw_i|^q\,dx\)^{\frac{1}{q}}
\le
c\(\bint_{B_i^{(20)}}
H(x,\ov{\M} Dw_i)\,dx\)^{\frac{1}{p}}
\end{align}
for some $c=c(\data,K)>0$.
Now, going back to $({\rm{I}})$ in \eqref{1134}, Young's inequality implies
\begin{align}\label{1144}
({\rm{I}})
\le
c\bint_{B_i^{(10)}} a_{i,o}|\ov{\M}Dw_i|^q\,dx
+
c\bint_{B_i^{(10)}} a_{i,o}|\ov{\M}Dv_i|^q\,dx
\end{align}
for some $c=c(n,p,q,\nu,L)>0$.
Then the first integral on the right-hand side in \eqref{1144} is estimated as follows:
\begin{align}\label{1145}
\bint_{B_i^{(10)}} a_{i,o}|\ov{\M}Dw_i|^q\,dx
&\le
cr_{i}^\alpha \bint_{B_i^{(10)}} |\ov{\M}Dw_i|^q\,dx\nonumber\\
&\le
cr_{i}^{\alpha}
\(\bint_{B_i^{(20)}}
\[H(x,\ov{\M} Dw_i)\]^{1+\sigma_3}\,dx\)^{\frac{q-p}{p(1+\sigma_3)}}
\bint_{B_i^{(20)}}
H(x,\ov{\M} Dw_i)\,dx\nonumber\\
&\le
cr_{i}^{\kappa_1}
\|H(\cdot,\ov{\M} Dw_i)\|_{L^{1+\sigma_3}(B_i^{(20)})}^{\frac{q-p}{p}}
\bint_{B_i^{(20)}}
H(x,\ov{\M} Dw_i)\,dx\nonumber\\
&\le
c
\norm{H(\cdot, \M Du)}_{L^{1+2\sigma_3}(B_i^{(40)})}^{\frac{q-p}{p}}
r_{i}^{\kappa_1}\bint_{B_i^{(40)}}
H(x,\M Du)\,dx
\end{align}
for $\sigma_3>0$ as in \eqref{1126-1},
$\kappa_1:=\alpha-\frac{n(q-p)}{p(1+\sigma_3)}>0$ and $c=c(\data,K)>0$. Here
we used \eqref{1142}, \eqref{1143}, H\"older's inequality,
\eqref{norm-w}, \eqref{1124}, and \eqref{lavrentiev}.
Using $B_i^{(40)}\subset B_R\subset\Omega_{2\hat{\tau}}$ and Proposition~\ref{prop:u-higher}, 
this implies
\begin{align}\label{XX-MDw}
\bint_{B_i^{(10)}} a_{i,o}|\ov{\M}Dw_i|^q\,dx
&\le
cr_{i}^{\kappa_1}
\bint_{B_i^{(40)}}
H(x,\M Du)\,dx
\end{align}
for some $c=c(\data,K,\gamma,\hat{\tau},\|H(\cdot,\M F)\|_{L^{\gamma}(\Omega_{\hat{\tau}})})>0$.
From $r_i\le1$, the previous inequality and \eqref{1124} yield
\begin{align}\label{XX-H-xo-MDw}
\bint_{B_i^{(10)}}H(x_{i,o},\ov{\M} Dw_i)
\le
c\bint_{B_i^{(40)}}
H(x,\M Du)\,dx
\end{align}
for some $c=c(\data,K,\gamma,\hat{\tau},\|H(\cdot,\M F)\|_{L^{\gamma}(\Omega_{\hat{\tau}})})>0$.

Next, for the last integral in \eqref{1144},
we observe that
\begin{align*}
\frac{q^2}{p^2}\le\bigg(1+\frac{\alpha}{n}\bigg)^2<\frac{n}{n-2\beta}<\frac{n}{n-2\alpha}
\mbox{($=\infty$ when $\alpha=1$ and $n=2$)}
\end{align*}
for $\beta\in\left(\frac{n\alpha(\alpha+2n)}{2(n+\alpha)^2},\alpha\right)$,
which follows from Lemma~\ref{lem:w-reg}
$$H(x_{i,o},\ov{\M}Dw_i)\in L^{\frac{q}{p}}\(B_i^{(10)}\).$$
This leads to
\begin{align}\label{1147}
\bint_{B_i^{(10)}}
\[H(x_{i,o},\ov{\M} Dv_i)\]^{q/p}\,dx
\le
c\bint_{B_i^{(10)}}
\[H(x_{i,o},\ov{\M} Dw_i)\]^{q/p}\,dx
\end{align}
for some $c=c(n,p,q,\nu,L,\Lambda)>0$ in view of Lemma~\ref{lem:v-cz}.
In addition, we have 
\begin{align}\label{1147-1146}
\bint_{B_i^{(10)}} |\ov{\M}Dw_i|^q\,dx
\le
c
r_i^{-\frac{n(q-p)}{p(1+\sigma_3)}}
\bint_{B_i^{(40)}}
H(x,\M Du)\,dx
\end{align}
from the calculations in \eqref{1145},
and
\begin{align} \label{1146}
\bint_{B_i^{(10)}} |\ov{\M}Dw_i|^{q^2/p}\,dx
&\le
c\(\bint_{B_i^{(20)}}
H(x,\ov{\M} Dw_i)\,dx\)^{q^2/p^2}
\end{align}
from Lemma~\ref{lem:w-higher}, 
where $c=c(\data,K,\gamma,\hat{\tau},\|H(\cdot,\M F)\|_{L^{\gamma}(\Omega_{\hat{\tau}})})>0$.
Using \eqref{1147}-\eqref{1146} the last integral in \eqref{1144} is computed as follows:
\begin{align}\label{1148}
&\bint_{B_i^{(10)}} a_{i,o}|\ov{\M}Dv_i|^q\,dx\nonumber\\
&\le
cr_{i}^\alpha
\bint_{B_i^{(10)}}
\[H(x_{i,o},\ov{\M} Dv_i)\]^{q/p}\,dx\nonumber\\
&\le
cr_{i}^{\kappa_1}
\bint_{B_i^{(40)}}
H(x,\M Du)\,dx\nonumber\\
&\quad\quad+
cr_{i}^{\alpha(1+q/p)}
\(\bint_{B_i^{(20)}}
\[H(x,\ov{\M} Dw_i)\]^{1+\sigma_3}\,dx\)^{\frac{1}{1+\sigma_3}\(\frac{q^2}{p^2}-1\)}
\bint_{B_i^{(20)}}
H(x,\ov{\M} Dw_i)\,dx\nonumber\\
&\le
cr_{i}^{\kappa_1}
\bint_{B_i^{(40)}}
H(x,\M Du)\,dx
+
cr_{i}^{\kappa_2}
\norm{H(\cdot, \ov{\M} Dw_i)}_{L^{1+\sigma_3}(B_i^{(20)})}^{q^2/p^2-1}
\bint_{B_i^{(20)}}
H(x,\ov{\M} Dw_i)\,dx\nonumber\\
&\le
cr_{i}^{\kappa_1}
\bint_{B_i^{(40)}}
H(x,\M Du)\,dx
+
cr_{i}^{\kappa_2}
\norm{H(\cdot, \M Du)}_{L^{1+2\sigma_3}(B_i^{(40)})}^{q^2/p^2-1}
\bint_{B_i^{(20)}}
H(x,\ov{\M} Dw_i)\,dx\nonumber\\
&\le
cr_{i}^{\kappa_1}
\bint_{B_i^{(40)}}
H(x,\M Du)\,dx
\end{align}
for $\sigma_3>0$ as in \eqref{norm-w}, $\kappa_1 > 0$ as in
\eqref{1145},
$\kappa_2=\(1+\frac{q}{p}\)\(\alpha-\frac{n}{1+\sigma_3}\(\frac{q}{p}-1\)\)>0$,
and $c=c(\data,K,\gamma,\hat{\tau},\|H(\cdot,\M F)\|_{L^{\gamma}(\Omega_{\hat{\tau}})})>0$. 
Here we have also used \eqref{1142}, \eqref{norm-w},
\eqref{1124}, $B_i^{(40)}\subset \Omega_{2\hat{\tau}}$, Proposition~\ref{prop:u-higher}
and the fact that $0<r_i<1$ and $\kappa_1<\kappa_2$.
We now combine \eqref{1134}, \eqref{1144}, \eqref{XX-MDw} and
\eqref{1148} to discover that
\begin{align*}
&\bint_{B_i^{(10)}} \(|V_p(\ov{\M}Dw_i)-V_p(\ov{\M}Dv_i)|^2
+ a(x)|V_q(\ov{\M}Dw_i)-V_q(\ov{\M}Dv_i)|^2 \)\,dx  \nonumber\\
&\le
c_2r_{i}^{\kappa_1}\bint_{B_i^{(40)}}H(x,\M Du)\,dx
\end{align*}
for some $c_2=c_2(\data,K,\gamma,\hat{\tau},\|H(\cdot,\M F)\|_{L^{\gamma}(\Omega_{\hat{\tau}})})>0$.

Consequently, in both cases of \eqref{1135} and \eqref{1136}, we see
that
\begin{align}\label{w-v}
&\bint_{B_i^{(10)}}
\(|V_p(\ov{\M}Dw_i)-V_p(\ov{\M}Dv_i)|^2
+
a(x)|V_q(\ov{\M}Dw_i)-V_q(\ov{\M}Dv_i)|^2\)\,dx\nonumber\\
&\le \(\frac{c_1}{K}+c_2r_{i}^{\kappa_1}\)\bint_{B_i^{(40)}}H(x,\M
Du)\ dx,
\end{align}
where $c_1=c_1(\data)>0$ and $c_2=c_2(\data, K,\gamma,\hat{\tau},\|H(\cdot,\M F)\|_{L^{\gamma}(\Omega_{\hat{\tau}})})>0$. At the same
time, considering $K=30$ fixed in \eqref{1135} and \eqref{1136}, we
obtain from \eqref{1132}, \eqref{1142-1} and \eqref{XX-H-xo-MDw} that
\begin{align*}
\bint_{B_i^{(10)}}H(x_{i,o},\ov{\M}Dv_i)\,dx
\le
c\bint_{B_i^{(10)}}H(x_{i,o},\ov{\M}Dw_i)\,dx
\le
c\bint_{B_i^{(40)}}
H(x,\M Du)\,dx
\end{align*}
for some $c=c(\data,\gamma,\hat{\tau},\|H(\cdot,\M F)\|_{L^{\gamma}(\Omega_{\hat{\tau}})})>0$ and hence, by Lemma~\ref{lem:v-lip}, 
\begin{align}\label{1154}
\sup_{B_i^{(5)}}H(x,\ov{\M}Dv_i)
\le
\sup_{B_i^{(5)}}H(x_{i,o},\ov{\M}Dv_i)
\le c\bint_{B_i^{(40)}} H(x,\M Du)\,dx.
\end{align}

\

We finally combine \eqref{lambda-bdd}, \eqref{u-h}, \eqref{h-k},
\eqref{k-w} and \eqref{w-v} to discover
\begin{align}\label{1155}
\bint_{B_i^{(10)}}
\(|V_p(\M Du)-V_p(\ov{\M}Dv_i)|^2
+
a(x)|V_q(\M Du)-V_q(\ov{\M}Dv_i)|^2\)\,dx
\le
S\lambda,
\end{align}
where
\begin{align}
\label{1157-2} S\equiv S(\eps,\delta,K,r_o) =
c_o\eps+c_\eps\delta^\sigma+\frac{\bar{c}}{K}+c_*r_{o}^{\kappa_1}
\end{align}
for some positive numbers $c_o=c_o(\data)$, $\sigma=\sigma(\data)$,
$c_\eps=c_\eps(\eps,\data)$, $\bar{c}=\bar{c}(\data)$ and
$c_*=c_*(\data,K,\gamma,\hat{\tau},\|H(\cdot,\M F)\|_{L^{\gamma}(\Omega_{\hat{\tau}})})$. In addition, from \eqref{lambda-bdd} and
\eqref{1154}, we obtain
\begin{align}\label{1156}
\sup_{B_i^{(5)}}H(x,\ov{\M}Dv_i)
\le
c_l\lambda
\end{align}
for some $c_l=c_l(\data,\gamma,\hat{\tau},\|H(\cdot,\M F)\|_{L^{\gamma}(\Omega_{\hat{\tau}})})>0$.

\

{\it Step 3: Estimates over upper level sets.} We recall that
\begin{align*}
H(x,\M Du) \le 2\(|V_p(\M Du)-V_p(\ov{\M}Dv_i)|^2 + a(x)|V_q(\M
Du)-V_q(\ov{\M}Dv_i)|^2\) + 2H(x,{\ov{\M}}Dv_i)
\end{align*}
to find from \eqref{1155} and \eqref{1156} that
\begin{align*}
\int_{B_i^{(5)}\cap\{H(x,\M Du)>4c_l\lambda\}} H(x,\M Du)\,dx \le
40^n S\lambda \left|B_i^{(1)}\right|.
\end{align*}
In view of \eqref{1107-1} we have that
\begin{align*}
|B_i^{(1)}|
\le
\frac{2}{\lambda}
\int_{B_i^{(1)}\cap \{H(x,\M Du)>\lambda/4\}}
H(x,\M Du)\,dx
+
\frac{2}{\lambda}
\int_{B_i^{(1)}\cap \{H(x,\M F)>\delta\lambda/4\}}
\frac{1}{\delta}H(x,\M F)\,dx,
\end{align*}
which yields
\begin{align*}
&\int_{B_i^{(5)}\cap\{H(x,\M Du)>4c_l\lambda\}}
H(x,\M Du)\,dx\nonumber\\
&\le
80^n S
\int_{B_i^{(1)}\cap \{H(x,\M Du)>\lambda/4\}}
H(x,\M Du)\,dx
+
80^n S
\int_{B_i^{(1)}\cap \{H(x,\M F)>\delta\lambda/4\}}
\frac{1}{\delta}H(x,\M F)\,dx.
\end{align*}
Recalling \eqref{1109-1} and using the fact that balls
$\{B_i^{(1)}\}$ are mutually disjoint, we conclude
\begin{align}\label{1159}
&\int_{E(\rho_1,4c_l\lambda)}
H(x,\M Du)\,dx\nonumber\\
&\le
80^n S
\int_{E(\rho_2,\lambda/4)}
H(x,\M Du)\,dx
+
80^n S
\int_{B_{\rho_2}\cap \{H(x,\M F)>\delta\lambda/4\}}
\frac{1}{\delta}H(x,\M F)\,dx
\end{align}
for any $\lambda\ge\lambda_0$.

\

{\it Step 4: Final estimate.}
We now prove the final estimate.
We first define truncations
\begin{align*}
[H(x,\M Du)]_t:=
\min \{H(x,\M Du),t\}
\quad\mbox{for $t>0$}.
\end{align*}
We remark for $t>\lambda$ that $[H(x,\M Du)]_t>\lambda$ if and only if $H(x,\M Du)>\lambda$.
Using Fubini's theorem, we calculate for any $t>\lambda_0$
\begin{align}\label{1164-d}
&\int_{B_{\rho_1}}
[H(x,\M Du)]_t^{\gamma-1}H(x,\M Du)\,dx\nonumber\\
&\le
(4c_l\lambda_0)^{\gamma-1}\hspace{-0.1cm}
\int_{B_{\rho_2}}
H(x,\M Du)\,dx
+
\underbrace{(\gamma-1)(4c_l)^{\gamma-1}\hspace{-0.1cm}
\int_{\lambda_0}^{4t}
\lambda^{\gamma-2}
\int_{E(\rho_1,4c_l\lambda)}
H(x,\M Du)\,dxd\lambda}_{=:{\rm{(J)}}}.
\end{align}
For the estimate of $\rm{(J)}$, we use \eqref{1159}, a change of variables and Fubini's theorem, to
obtain
\begin{align*}
{\rm{(J)}}
&\le
c(n,\gamma) Sc_l^{\gamma-1}
\int_{0}^{t}
\lambda^{\gamma-2}
\int_{B_{\rho_2}\cap\{[H(x,Du)]_t>\lambda\}}
H(x,\M Du)\,dxd\lambda\nonumber\\
&\quad\quad+
c(n,\gamma)Sc_l^{\gamma-1}
\int_{0}^{\infty}
\lambda^{\gamma-2}
\int_{B_{\rho_2}\cap \{H(x,\M F)>\delta\lambda/4\}}
\frac{1}{\delta}H(x,\M F)\,dxd\lambda\nonumber\\
&\le
c(n,\gamma) Sc_l^{\gamma-1}
\int_{B_{\rho_2}}
[H(x,\M Du)]_t^{\gamma-1}H(x,\M Du)\,dx\nonumber\\
&\quad\quad+
c(n,\gamma)Sc_l^{\gamma-1}
\int_{B_{\rho_2}}
\frac{1}{\delta^\gamma}[H(x,\M F)]^\gamma\,dx.
\end{align*}
Inserting this into \eqref{1164-d}, we deduce
\begin{align}\label{1164}
&\int_{B_{\rho_1}}
[H(x,\M Du)]_t^{\gamma-1}H(x,\M Du)\,dx\nonumber\\
&\le
\tilde{c}c_l^{\gamma-1}\lambda_0^{\gamma-1}
\int_{B_{\rho_2}}
H(x,\M Du)\,dx
+
\tilde{c} c_l^{\gamma-1} S
\int_{B_{\rho_2}}
[H(x,\M Du)]_t^{\gamma-1}H(x,\M Du)\,dx\nonumber\\
&\quad\quad+
\tilde{c}c_l^{\gamma-1} \delta^{-\gamma}S
\int_{B_{\rho_2}}
[H(x,\M F)]^\gamma\,dx
\end{align}
for some $\tilde{c}=\tilde{c}(n,\gamma)>0$. Recalling
\eqref{1157-2}, we now select $\eps,\delta,K$ and $r_o$ in order to
get
\begin{align*}
\tilde{c} c_l^{\gamma-1} S\le\frac12.
\end{align*}
We first take
\begin{align}\label{1169}
 K=\max\{8\tilde{c}c_l^{\gamma-1}\ov{c},30\}
 \quad\mbox{and}\quad
 \eps=\frac{1}{8\tilde{c}c_l^{\gamma-1}c_o}
 \end{align}
so that constants $c_*$ and $c_\eps$ in \eqref{1157-2} are
determined depending on $\data,\gamma,\hat{\tau}$ and $\|H(\cdot,\M F)\|_{L^{\gamma}(\Omega_{\hat{\tau}})}$. We next choose $\delta$ and $r_o$
such that
\begin{align}\label{1170}
0<\delta\le\min\left\{\(\frac{1}{8\tilde{c}c_l^{\gamma-1}c_\eps}\)^{1/\sigma},W\right\}
\quad\mbox{and}\quad
0<r_o
\le
\min\left\{\(\frac{1}{8\tilde{c}c_l^{\gamma-1}c_*}\)^{1/\kappa_1},1\right\}.
\end{align}
With the choices in \eqref{1169} and \eqref{1170}, we recall
\eqref{1104} and \eqref{1164} to have
\begin{align*}
&\int_{B_{\rho_1}}
[H(x,\M Du)]_t^{\gamma-1}H(x,\M Du)\,dx\nonumber\\
&\le
\tfrac{1}{2}
\int_{B_{\rho_2}}
[H(x,\M Du)]_t^{\gamma-1}H(x,\M Du)\,dx
+
c
\int_{B_{R}}
[H(x,\M F)]^\gamma\,dx\nonumber\\
&\quad\quad+
c
\(\frac{40R}{\rho_2-\rho_1}\)^{n(\gamma-1)}
\lambda_R^{\gamma-1}
\int_{B_{R}} H(x,\M Du)\,dx
\end{align*}
for any $R/2\le\rho_1<\rho_2\le R$, where $c=c(\data,\gamma,\hat{\tau},\|H(\cdot,\M F)\|_{L^{\gamma}(\Omega_{\hat{\tau}})})>0$. 
We now apply Lemma~\ref{lem:st} to obtain
\begin{align*}
\int_{B_{R/2}}
[H(x,\M Du)]_t^{\gamma-1}H(x,\M Du)\,dx
\le
c\lambda_R^{\gamma-1}
\int_{B_{R}} H(x,\M Du)\,dx
+
c\int_{B_{R}}
[H(x,\M F)]^\gamma\,dx.
\end{align*}
Letting $t\rightarrow\infty$, using Fatou's lemma and recalling
\eqref{1105-2}, we finally conclude
\begin{align*}
\bint_{B_{R/2}}
[H(x,\M Du)]^{\gamma}\,dx
\le
c\(\bint_{B_{R}} H(x,\M Du)\,dx\)^\gamma
+
c\bint_{B_{R}}
[H(x,\M F)]^\gamma\,dx
\end{align*}
for any $0<R\le r_o$ with $r_o=r_o(\data,\gamma,\hat{\tau},\|H(\cdot,\M F)\|_{L^{\gamma}(\Omega_{\hat{\tau}})})\in(0,1]$ as in
\eqref{1170}, where $c=c(\data,\gamma,\hat{\tau},\|H(\cdot,\M F)\|_{L^{\gamma}(\Omega_{\hat{\tau}})})>0$.
\end{proof}

\subsection*{Conflict of interest}
The authors declare that there is no conflict of interest.

\subsection*{Data availability}
Data sharing not applicable to this article as no datasets were generated or analyzed during the current study.

\subsection*{Acknowledgements}
The authors sincerely thank the anonymous reviewer for their careful reading of the manuscript and for several important comments and suggestions that improved both the exposition and the mathematical rigor of the paper.


\providecommand{\bysame}{\leavevmode\hbox to3em{\hrulefill}\thinspace}
\providecommand{\MR}{\relax\ifhmode\unskip\space\fi MR }
\providecommand{\MRhref}[2]{%
  \href{http://www.ams.org/mathscinet-getitem?mr=#1}{#2}
}
\providecommand{\href}[2]{#2}

\end{document}